\documentclass[sn-mathphys,Numbered]{sn-jnl}%
\usepackage{graphicx}%
\usepackage{multirow}%
\usepackage{amsmath,amssymb,amsfonts}%
\usepackage{amsthm}%
\usepackage{mathrsfs}%
\usepackage[title]{appendix}%
\usepackage{xcolor}%
\usepackage{textcomp}%
\usepackage{manyfoot}%
\usepackage{booktabs}%
\usepackage{algorithm}%
\usepackage{algorithmicx}%
\usepackage{algpseudocode}%
\usepackage{listings}%
\usepackage{hyperref}
\usepackage{cleveref}
\usepackage{enumitem}
\usepackage{comment}
\usepackage{ulem}

\theoremstyle{thmstyleone}%
\newtheorem{theorem}{Theorem}
\newtheorem{proposition}[theorem]{Proposition}%
\newtheorem{lemma}[theorem]{Lemma}

\theoremstyle{thmstyletwo}%
\newtheorem{example}{Example}%
\newtheorem{remark}{Remark}%

\theoremstyle{thmstylethree}%
\newtheorem{definition}{Definition}%

\theoremstyle{thmstylefour}%
\newtheorem{assumption}{Assumption}%

\begin{document}

\title[Sampling strategies for expectation values within the Herman--Kluk approximation]{Sampling strategies for expectation values within the Herman--Kluk approximation}


\author[1]{\fnm{Fabian} \sur{Kröninger}}\email{fabiankroninger@epfl.ch}

\author[2]{\fnm{Caroline} \sur{Lasser}}\email{classer@ma.tum.de}

\author[1]{\fnm{Ji\v{r}\'{i} J.~L.} \sur{ Van\'{i}\v{c}ek}}\email{jiri.vanicek@epfl.ch}

\affil[1]{Laboratory of Theoretical Physical Chemistry, Institut des Sciences et Ingénierie Chimiques, Ecole Polytechnique Fédérale de Lausanne (EPFL), CH-1015 Lausanne, Switzerland}

\affil[2]{ School of Computation, Information and Technology, Department of Mathematics, Technische Universität München, Germany }

\abstract{When computing quantum-mechanical observables, the ``curse of dimensionality'' limits the naive approach that uses the quantum-mechanical wavefunction.
The semiclassical Herman--Kluk propagator mitigates this curse by employing a grid-free ansatz to evaluate the expectation values of these observables.
Here, we investigate quadrature techniques for this high-dimensional and highly oscillatory propagator.
In particular, we analyze Monte Carlo quadratures using three different initial sampling approaches. 
The first two, based either on the Husimi density or its square root, are independent of the observable whereas the third approach, which is new, incorporates the observable in the sampling to minimize the variance of the Monte Carlo integrand at the initial time. 
We prove sufficient conditions for the convergence of the Monte Carlo estimators and provide convergence error estimates.
The analytical results are validated by numerical experiments in various dimensions on a harmonic oscillator and on a Henon-Heiles potential with an increasing degree of anharmonicity.}
\keywords{Schrödinger equation; semiclassical Herman--Kluk; quantum-mechanical observables; Monte Carlo quadrature; frozen Gaussian approximation}


\maketitle

\section{Introduction}\label{Section:1}
Molecular quantum dynamics is a vibrant, interdisciplinary field of research dedicated to gaining insights into chemical and physical phenomena. To simulate it, one must solve the time-dependent Schrödinger equation 
\begin{align}\label{EQ:TDSE}
    i\epsilon \partial_t \psi(t,x) = \left(-\frac{\epsilon^2}{2} \Delta + V(x) \right) \psi(t,x), \quad \psi(0,x) = \psi_0(x)\in L^2(\mathbb{R}^D)
\end{align}
for nuclei on an electronic potential energy surface $V$ that arises from the Born--Oppenheimer approximation \cite{Born_Oppenheimer:1927}.
Here, $0<\epsilon\ll 1$ denotes the semiclassical parameter, $\Delta:= \sum_{j=1}^D \partial^2_{x_j}$ is the Laplacian and $V:\mathbb{R}^D \rightarrow \mathbb{R}$ is a smooth potential of sub-quadratic growth \cite[Chapter 11.1]{Robert_Combescure:2021} such that the Hamiltonian $\hat{H}=-\frac{\epsilon^2}{2}\Delta + V$ is a self-adjoint operator when equipped with an appropriate domain. The unitary evolution operator
\begin{align} \label{EQ:Unitary_evolution_operator}
\hat{U}_t := \exp{(-i\hat{H}t/\epsilon)} 
\end{align}
then gives the solution to \eqref{EQ:TDSE} as $\psi(t) = \hat{U}_t\psi_0$.

The numerical solution of the time-dependent Schrödinger equation is challenging.
In most applications, conventional grid-based integration methods are not feasible due to the large dimension $D$ of the configuration space of a molecular system. 
In addition, the high oscillations induced by the small parameter $\epsilon$ exacerbate this curse of dimensionality.

To overcome these issues, semiclassical methods \cite{Miller:2001, book_Tannor:2007,  book_Lubich:2008, book_Heller:2018, Lasser_Lubich:2020} have been developed using a priori knowledge of the solution's behaviour. 
Single-particle semiclassical methods, such as the thawed Gaussian approximation \cite{Heller:1975}, are limited in their applicability because some quantum effects, such as wavepacket splitting, cannot be described by these methods. 
Multi-particle methods often remain accurate and provide further qualitative insights. 
The Herman--Kluk approximation \cite{Herman_Kluk:1984}, which refines Heller's frozen Gaussian approximation \cite{Heller:1981}, remains one of the most accurate multi-particle semiclassical methods.

Within the Herman--Kluk approximation, the solution of the time-dependent Schrödinger equation \eqref{EQ:TDSE} is approximated by a high-dimensional and highly oscillatory integral over the phase space $\mathbb{R}^{2D}$ which has the general form
\begin{align}\label{EQ:HK_introduction}
   \psi(t,x) \approx \int_{\mathbb{R}^{2D}}  a_t(z) e^{i\phi_t(x,z)/\epsilon}   \,dz.
\end{align}
The integrand is time-dependent and the complex-valued functions $a_t$ and $\phi_t$ are obtained by solving ordinary differential equations.
Integral \eqref{EQ:HK_introduction} is motivated by the continuous superposition of time-dependent Gaussian wavepackets
\begin{align}\label{EQ:Gaussian}
    g_{z(t)}(x)=\left(\frac{\det \Gamma}{\pi^D\epsilon^D}  \right)^{1/4}\exp{\left(-\frac{1}{2\epsilon} (x-q(t))^T \Gamma (x-q(t)) +\frac{i}{\epsilon}  p(t)^T(x-q(t))  \right)}
\end{align} that have a fixed, time-independent (``frozen''), symmetric, and positive-definite width matrix $\Gamma\in\mathbb{R}^{D\times D}$, and are centered at a phase-space point $z(t)=(q(t),p(t))\in\mathbb{R}^{2D}$.
The time-dependent position and momentum parameters $q(t)$ and $p(t)$ obey Hamilton's equations of motion $\dot{q}=p$ and $\dot{p}=-\nabla V(q)$, whereas the lack of change in width $\Gamma$ is compensated by a time-dependent prefactor.
Due to the frozen Gaussian ansatz, the function $e^{i\phi_t(\cdot,z)/\epsilon}$ in \eqref{EQ:HK_introduction} will be of Schwartz class on the configuration space $\mathbb{R}^D$.

The use of the wavefunction per se is limited by the dimension of the configuration space $\mathbb{R}^D$.
In typical applications, where $D\gg 1$, further insights into a quantum system can
be obtained only by circumventing the evaluation of the wavefunction and, instead, by directly evaluating expectation values of quantum-mechanical observables. These expectation values provide information not only about position, momentum and kinetic, potential and total energies but also about the norm of the state.
For a normalized state $\psi$, the expectation value of a self-adjoint operator $\hat{A}$ on $L^2(\mathbb{R}^D)$ is given by the inner product
\begin{align}\label{EQ:Ex_value_introduction}
    \langle \psi , \hat{A} \psi \rangle = \int_{\mathbb{R}^D} \overline{\psi(x)} (\hat{A} \psi)(x) \,dx.
\end{align}
Conventionally, one must obtain a solution of the time-dependent Schrödinger equation \eqref{EQ:TDSE} and then compute the overlap \eqref{EQ:Ex_value_introduction} using a numerical quadrature. 
Inserting approximation \eqref{EQ:HK_introduction} into the inner product \eqref{EQ:Ex_value_introduction} leads to
\begin{align}\label{EQ:HK_ex_value_introduction}
    \langle \psi(t), \hat{A} \psi(t) \rangle \approx\int_{\mathbb{R}^{4D}} \overline{a_t(y)} a_t(z)\langle e^{i\phi_t(y)/\epsilon}, \hat{A} e^{i\phi_t(z)/\epsilon} \rangle \,d(y,z).
\end{align}
This integral acts on the double phase space $\mathbb{R}^{2D} \otimes \mathbb{R}^{2D}  \simeq \mathbb{R}^{4D} $. 
As we shall see in \Cref{Sec:Herman_Kluk}, the special form of $\phi_t$ due to the frozen Gaussian ansatz ensures the well-definedness of the inner product $\langle e^{i\phi_t(y)/\epsilon}, \hat{A} e^{i\phi_t(z)/\epsilon} \rangle$ and significantly simplifies it. 
In particular, an additional numerical quadrature is often unnecessary, as explicit formulas are known. 
Instead of two separate Herman--Kluk integrals on $\mathbb{R}^{2D}$, the expectation value is interpreted as a weighted integral on a double phase space $\mathbb{R}^{4D}$ with respect to a product measure and it can be computed with a single numerical quadrature.
The notion of double phase space also appears in other contexts of mathematical physics, for example, in the application of the Weyl propagator \cite{Ozorio_Ingold:2021, Littlejohn:1990} or forward-backward semiclassical propagation \cite{Antipov_Ananth:2015, Makri_Miller:2002, Ozorio_Brodier:2006, Saraceno:2016, Church_Ananth:2017}.

\subsection{Previous research}
The Herman--Kluk propagator has been used extensively in theoretical chemistry. 
Several groups have applied further approximations to make the Herman--Kluk propagator computationally more feasible, see \cite{Walton_Manolopoulos:1996, Elran_Kay:1999a, Ceotto_Atahan:2009a, Makri:2011, Antipov_Ananth:2015} for some exemplary contributions.
In \cite{Zimmermann_Vanicek:2013}, the authors investigated sampling approaches for classical time autocorrelation functions and they proposed an observable-dependent sampling ensuring that the convergence of the estimator is independent of the dimensionality of the system and the underlying dynamics.

From a mathematical point of view, the Herman--Kluk propagator has been rigorously analyzed by several authors. 
Swart and Rousse \cite{Swart_Rousse:2008} and Robert \cite{Robert2010} proved that the Herman--Kluk propagator approximates the unitary evolution operator to the first-order of $\epsilon$. 
An extensive overview of single- and multi-particle semiclassical methods (including the Herman--Kluk propagator) and their error analysis can be found in \cite{Lasser_Lubich:2020}.

In the mathematical literature, the Herman--Kluk approximation is sometimes referred to as the frozen Gaussian approximation.
Lu and Yang apply it to high-frequency wave propagation \cite{Lu_Yang:2010, Lu_Yang:2012} and general linear strictly hyperbolic systems \cite{Lu_Yang:2012a}.
Delgadillo et al. generalized the Herman--Kluk propagator to a gauge-invariant frozen Gaussian approximation for the Schrödinger equation with periodic potentials  \cite{Delgadillo_Yang:2016, Delgadillo_Yang:2018}. 
In \cite{Lasser_Sattlegger:2017}, the numerical discretization of the Herman--Kluk wavefunction in both time and phase space was analyzed and the special form of the expectation value \eqref{EQ:HK_ex_value_introduction} was introduced. 
Moreover, various wavefunction sampling approaches were investigated in \cite{Kroeninger_Lasser_Vanicek:2023}.
Xie and Zhou also investigated mesh-free discretization methods for wavefunctions and expectation values \cite{Xie_Zhou:2021}. They introduced an estimator for expectation values that scales quadratically with the number of quadrature points, whereas our approach scales linearly.
Huang et al. combined the Herman--Kluk approximation with surface hopping to compute the wavefunction and expectation values in a nonadiabatic regime \cite{Huang_Zhou:2022}.

\subsection{Main results}
Owing to the frozen Gaussian ansatz, the Herman--Kluk expectation value \eqref{EQ:HK_ex_value_introduction} can be perfectly combined with probabilistic discretization methods such as the Monte Carlo quadrature.
Common Monte Carlo approaches originate from the Husimi and ``sqrt-Husimi'' (absolute value of an inner product of a Gaussian wavepacket with the initial state $\psi_0$) densities. 
While the Husimi density is well-defined for all $\psi_0$ coming from $L^2(\mathbb{R}^D)$, the sqrt-Husimi density requires a restriction on the domain of the initial states.
Because both of these densities act only on a single phase-space $\mathbb{R}^{2D}$, the quadrature points of the double phase-space
$\mathbb{R}^{4D}$ are then generated by sampling both individual phase spaces either from the Husimi or sqrt-Husimi densities.
For the most important quantum-mechanical operators $\hat{A}$, we show that the Monte Carlo integrand based on the Husimi density has an unbounded variance while choosing the sqrt-Husimi density leads to finite values.
Moreover, we show examples in which the variance of the sqrt-Husimi approach has an exponential dependence on the spatial dimension $D$ of the system.
The convergence of the estimators can be further improved by including operator $\hat{A}$ in the sampling density.
We prove that such a choice minimizes the variance of the Monte Carlo integrand at the initial time.
Sampling from a probability density that incorporates the observable was also investigated for classical time autocorrelation functions \cite{Zimmermann_Vanicek:2013}. The authors show that the convergence of the Monte Carlo estimator is independent of the system’s dimensionality and underlying dynamics. However, for Herman--Kluk expectation values we provide an example where the variance still has an exponential dependence on the dimension $D$. In fact, it is not uncommon for the convergence of Monte Carlo quadrature to depend on the dimension of the problem \cite[Section 12.2]{Henning_Kersting:2022}.

Because of the inclusion of operator $\hat{A}$, the new sampling density may only be known up to a constant value, making the standard Monte Carlo estimator inaccurate.
To overcome this issue, we introduce self-normalizing estimators, which are based on a weighted average, whereas the standard Monte Carlo quadrature is non-weighted.
Specifically, we approximate the Herman--Kluk expectation value \eqref{EQ:HK_ex_value_introduction} using an estimator in the form
\begin{align}
\begin{split}
   &\int_{\mathbb{R}^{4D}} \overline{a_t(y)} a_t(z)\langle e^{i\phi_t(y)/\epsilon}, \hat{A} e^{i\phi_t(z)/\epsilon} \rangle \,d(y,z) 
   \\& \approx \frac{1}{N}\sum_{j=1}^N\left[\overline{a_t(y_j)} a_t(z_j)\langle e^{i\phi_t(y_j)/\epsilon}, \hat{A} e^{i\phi_t(z_j)/\epsilon} \rangle \frac{W(y_j,z_j)}{\rho_1(y_j,z_j)}\right] /\sum_{j=1}^N W(y_j,z_j),
\end{split}
\end{align}
with a weight $W=\rho_1/\rho_2$, two probability densities $\rho_1$, $\rho_2$ and samples $(y_j,z_j)$ that are independent and identically distributed with respect to $\rho_2$.
We show that this estimator converges to the expectation value \eqref{EQ:HK_ex_value_introduction} as long as samples $(y_j,z_j)$ can be generated from $\rho_2$. In particular, it is sufficient to know the density $\rho_2$ only up to a constant factor.
Moreover, we prove that the bias of the estimator vanishes as $N\to \infty$ for any choice of the weight $W$ and that under certain conditions, the bias, variance, and mean squared error of the estimator are in the order of $\mathcal{O}(N^{-1})$.
We investigate several choices of weight $W$ based on the previously discussed sampling approaches. 

By combining Markov Chain Monte Carlo algorithms with self-normalizing estimators, we propose an advanced, novel algorithm for the approximation of expectation values with linear scaling improving on known approaches with quadratic scaling.

\subsection{Outline of the paper}
The remainder of this paper is organized as follows. 
In the next section, we introduce the Herman--Kluk approximation and discuss its ability to compute the expectation values of observables $\hat{A}$. 
In \Cref{Sec:2}, we introduce the Monte Carlo quadrature and its crude Monte Carlo estimator as our main numerical tool for solving the high-dimensional and highly oscillatory integral \eqref{EQ:Ex_value_introduction} obtained from the Herman--Kluk approximation. 
We provide sufficient conditions for the convergence of the estimator and investigate the different sampling approaches.  
In \Cref{Sec:3}, we extend the crude Monte Carlo estimator to a self-normalizing version and analyze its properties in detail.
The general algorithm that computes \eqref{EQ:Ex_value_introduction} within the Herman--Kluk approximation is summarized in \Cref{Sec:4}. We complete the paper with numerical examples in several dimensions in \Cref{Sec:5} and a short discussion in \Cref{Sec:6}.

\subsection{Notation}
Throughout this paper, we denote the space of Schwartz functions on the configuration space $\mathbb{R}^D$ as $\mathcal{S}(\mathbb{R}^D)$.
These are smooth functions that decrease rapidly together with all their derivatives.
Moreover, we denote a point in the double phase space $\mathbb{R}^{4D}$ by  $w=(y,z) \in\mathbb{R}^{4D}$, consisting of $y$ and $z$ each originating from a single phase space $\mathbb{R}^{2D}$.

We restrict ourselves to the linear operators $\hat{A}=\mathrm{Id}$ (representing the norm squared), powers $\hat{q}^n$ and $\hat{p}^n$ for $n\in\mathbb{N}$ of position and momentum, kinetic energy $T(\hat{p})$, potential energy $V(\hat{q})$ and total energy $\hat{H}=T(\hat{p}) + V(\hat{q})$ as they are the most interesting operators in quantum mechanics. We denote the $D$-dimensional identity matrix by $\mathrm{Id}_D$.

For a random variable $X$ distributed with respect to probability density $\rho$, the expectation value and variance are denoted by
\begin{align}
    \mathbb{E}[X]= \int x\rho(x) \,dx \quad\text{and}\quad \text{Var}[X] = \int \vert x - \mathbb{E}[X] \vert^2 \rho(x) \,dx
\end{align}
and for an estimator $\mathbb{X}$ of $\mathbb{E}[X]$, the bias is defined as
\begin{align}
    \text{Bias}(\mathbb{X}) = \vert \mathbb{E}[\mathbb{X}] - \mathbb{E}[X] \vert.
\end{align}

\section{The Herman--Kluk propagator}\label{Sec:Herman_Kluk}
Similar to Heller \cite{Heller:1981}, Herman and Kluk \cite{Herman_Kluk:1984} argued that a quantum system cannot be described accurately by a single Gaussian wavepacket. Instead, they proposed using multiple Gaussians to solve the time-dependent Schrödinger equation \eqref{EQ:TDSE}.

\subsection{Integral representation}
The inversion formula of the Fourier-Bros-Iagolnitzer (FBI) transform provides a tool to represent a wavefunction in terms of a phase-space integral of Gaussians.

For $z\in\mathbb{R}^{2D}$, the FBI transform \cite[Chapter 3.1]{Martinez:2002} is defined as
\begin{align} \label{EQ:FBI_transform}
    \mathcal{T}: \mathcal{S}(\mathbb{R}^D) \rightarrow \mathcal{S}(\mathbb{R}^{2D}), \quad (\mathcal{T}\psi)(z) := (2\pi\epsilon)^{-D/2} \langle g_z, \psi \rangle,
\end{align}
where the Gaussian $g_z$ is parameterized as in \eqref{EQ:Gaussian}.
The FBI transform maps $\mathcal{S}(\mathbb{R}^D)$ continuously into $\mathcal{S}(\mathbb{R}^{2D})$ (see \cite[Proposition 3.1.6]{Martinez:2002})
and there is the inversion formula \cite[Proposition 5.1]{Lasser_Lubich:2020}
    \begin{align}\label{EQ:Inversion_FBI}
        \psi(x) = (2\pi\epsilon)^{-D} \int_{\mathbb{R}^{2D}} \langle g_z, \psi \rangle g_z(x) \,dz
    \end{align}
for all $\psi\in \mathcal{S}(\mathbb{R}^D)$. The FBI transform can be extended to an isometry from $L^2(\mathbb{R}^D)$ to $L^2(\mathbb{R}^{2D})$ (see \cite[Proposition 3.1.1]{Martinez:2002}) and, because the Schwartz space $\mathcal{S}$ is dense in $L^2$, the inversion formula \eqref{EQ:Inversion_FBI} extends to square-integrable functions.

Based on the continuous superposition \eqref{EQ:Inversion_FBI}, one writes 
\begin{align}
    (\hat{U}_t\psi_0)(x) = (2\pi\epsilon)^{-D} \int_{\mathbb{R}^{2D}} \langle g_z, \psi_0 \rangle (\hat{U}_t g_z)(x)\,dz,
\end{align}
which motivates the Herman--Kluk propagator
\begin{align} \label{EQ:Herman_Kluk_propagator}
\begin{split}
    (\hat{U}_t^{\rm HK} \psi_0)(x) &:=(2\pi\epsilon)^{-D} \int_{\mathbb{R}^{2D}}\langle g_z, \psi_0 \rangle  R_t(z) e^{iS_t(z)/\epsilon}  g_{z(t)}(x)  \,dz,
\end{split}
\end{align}
which is in the form of \eqref{EQ:HK_introduction} with

\begin{align}
    \begin{split}
    a_t(z)&= (2\pi\epsilon)^{-D}R_t(z) \langle g_z,\psi_0 \rangle  
    \\ \text{and} \quad \phi_t(x,z) &= \frac{i}{2} (x-q(t))^T \Gamma (x-q(t)) + p(t)^T(x-q(t)) + S_t(z).
    \end{split}
\end{align}
In particular, since $\Gamma$ is real symmetric and positive-definite, we have $e^{i\phi_t(\cdot, z)}\in L^2(\mathbb{R}^D)$ for any fixed $z\in\mathbb{R}^{2D}$.
Here, $z(t)=(q(t),p(t))\in\mathbb{R}^{2D}$ is a classical trajectory in the phase space
associated with the Hamiltonian function $h(q,p)= \vert p \vert^2/2+V(q)$ and $S_t(z)\in\mathbb{R}$ is the classical action along this trajectory. 
The Herman--Kluk prefactor $R_t(z)\in\mathbb{C}$ is defined as
\begin{align} \label{HK_prefactor}
    R_t(z) :=2^{-D/2} \det{\left( M_{qq} + \Gamma^{-1}  M_{pp}  \Gamma -i M_{qp}  \Gamma + i \Gamma^{-1} M_{pq} \right)}^{1/2},
\end{align}
where the $D\times D$ matrices $M_{\alpha \beta}$, with $ \alpha, \beta \in \{q,p\},$  are the blocks of the stability matrix $M(t) := \partial z(t)/\partial z$, and the branch of the square root of \eqref{HK_prefactor} is chosen continuously with respect to time.
For all phase-space points $z\in \mathbb{R}^{2D}$, the parameters $z(t)=(q(t),p(t))$, $S_t$, and $M(t)$ obey the ordinary differential equations (ODEs)
\begin{align} 
    \dot{z}(t) &= J \cdot \nabla h(z(t)), \label{EQ:Traj}
    \\ \dot{S}_t(z) &= \frac{1}{2}\vert p(t) \vert^2 - V(q(t)), \label{EQ:S}
    \\ \dot{M}(t) & = J \cdot \text{\rm{Hess }} h(z(t)) \cdot M(t) \label{EQ:Stab_mat}
\end{align}
with the symplectic matrix $J: = \begin{pmatrix}
    0 & \mathrm{Id}_D \\ -\mathrm{Id}_D & 0
\end{pmatrix} \in \mathbb{R}^{2D}$, the Hamiltonian function
\begin{align}
    h: \mathbb{R}^{2D} \rightarrow \mathbb{R}, (q,p) \mapsto \frac{1}{2}\vert p \vert^2 + V(q)
\end{align} and initial conditions $z(0)=z$,  $S_0(z) = 0$ and $M(0) = \mathrm{Id}_D$. 
\begin{remark}
    In the mathematical literature, spherical Gaussians (i.\,e., Gaussians with $\Gamma=\textup{Id}_D$) are often considered \cite{Lasser_Lubich:2020, Lasser_Sattlegger:2017, Swart_Rousse:2008}. 
    In chemical physics simulations, the width matrix $\Gamma$ is usually chosen such that the Gaussian is an eigenstate of the Hamiltonian $\hat{H}=T+\tilde{V}$, where the potential $\tilde{V}$ is a harmonic approximation of the real potential $V$ at a reference position $q_0$ (i.\,e. $\tilde{V}$ is equal to the second order Taylor expansion of $V$ around $q_0$) \cite{Ceotto_Atahan:2009a, Makri:2011}. 
    Hence, we keep width matrices real symmetric and positive-definite to have flexible initial data for chemical applications.
\end{remark}

Swart and Rousse \cite{Swart_Rousse:2008}, and Robert \cite{Robert2010} rigorously justified the Herman--Kluk propagator (see also \cite{Kay:2006,Lasser_Lubich:2020,Lu_Yang:2012a, Lu_Yang:2012}). More recently, discretizations with respect to time and phase space for numerical calculations have been analyzed \cite{Lasser_Sattlegger:2017, Kroeninger_Lasser_Vanicek:2023, Xie_Zhou:2021}.
In particular, the Herman--Kluk propagator has the following properties:
\begin{enumerate}
\item For a sub-quadratic potential $V$, the Herman--Kluk propagator $\hat{U}_t^{\rm HK}$ approximates the unitary evolution operator \eqref{EQ:Unitary_evolution_operator}
with an error of the order of $\epsilon$ \cite[Theorem 5.3]{Lasser_Lubich:2020}, that is,
\begin{align}\label{EQ:Norm_accuracy}
    \sup_{t\in[0,\tau]} \vert\vert\hat{U}_t - \hat{U}_t^{\rm HK}\vert\vert_{L^2\rightarrow L^2} \leq C(\tau) \epsilon,
\end{align}
where $\tau>0$ is a fixed time, and the constant $C(\tau)$ is independent of $\epsilon$ and depends only on $\tau$. 
    \item For potentials $V$ that are at most quadratic polynomials, the Herman--Kluk approximation is exact \cite[Corollary 5.9]{Lasser_Lubich:2020}.
    \item For a sub-quadratic potential $V$, the Herman--Kluk prefactor $R_t$ is bounded from above and from below \cite[Lemma 5.6]{Lasser_Lubich:2020}. For a fixed time $\tau>0$, there exist constants $c_1(\tau)> 0$ and $c_2(\tau)>0$ such that
    \begin{align}
       c_1(\tau) \leq  \vert R_t(z) \vert \leq c_2(\tau)
    \end{align}
    for all $z\in \mathbb{R}^{2D}$ and $t\in[0,\tau]$.
\end{enumerate} 

\subsection{The representation for expectation values}
In practice, due to dimension $D\gg 1$, one is often more interested in the expectation values of quantum-mechanical observables than the wavefunction itself.
These expectation values provide information about the norm and measurable quantities such as the positions, momenta, and energies, and can be used for comparison with reference solutions (if they are available).

The expectation value of a self-adjoint operator $\hat{A}$ on $L^2(\mathbb{R}^{D})$ (or some subspace of it) in the normalized state $\psi$ is given by the inner product
\begin{align}\label{EQ:exact_exp_value}
    \langle \psi, \hat{A} \psi \rangle = \int_{\mathbb{R}^{D}} \overline{\psi(x)} (\hat{A}\psi)(x)\,dx.
\end{align}
\begin{remark}
    Let $\psi(t)$ be the solution to the time-dependent Schrödinger equation \eqref{EQ:TDSE} or some approximation of it and let $\hat{\rho}(t) = \vert\psi(t) \rangle \langle \psi(t)\vert$ (in Dirac notation, see \cite[Section 3.12]{Hall:2013}) be the orthogonal projection onto the span of $\psi(t)$. The expectation value of a quantum mechanical operator $\hat{A}$ can be written as (see \cite[Proposition 19.10]{Hall:2013})
    \begin{align}\label{EQ:remark_exp_values}
         \langle \psi(t), \hat{A} \psi(t) \rangle = \rm{Tr} \it{[\hat{\rho}(t) \hat{A}]}.
    \end{align}
    In the chemical literature, often the more general time-correlation function
    \begin{align} \label{EQ:Time_correlation_function}
        C_{\hat{B}\hat{A}}(t) = \rm{Tr}\it{[\hat{B}(t)\hat{A}]}
    \end{align}
    of operators $\hat{A}$ and $\hat{B}(t)$ is of interest \cite{Pollak_Liu:2022, Miller:2001, Sun_Miller:2002}. To simplify the exposition, our focus in this paper is on expectation values \eqref{EQ:remark_exp_values}, but we believe that our results can be extended to the more general case \eqref{EQ:Time_correlation_function}.
\end{remark}

A conventional calculation of the expectation value \eqref{EQ:exact_exp_value} requires finding an (approximate) solution $\psi(t)$ of the time-dependent Schrödinger equation \eqref{EQ:TDSE} and evaluating the inner product using an additional numerical quadrature. 
This is clearly limited to low-dimensional situations.
However, the Herman--Kluk propagator provides a convenient way of computing the expectation values without evaluating the full Herman--Kluk wavefunction $\hat{U}_t^{\rm HK} \psi_0$ for some normalized $\psi_0\in L^2(\mathbb{R}^D)$. 
By considering a double phase space $\mathbb{R}^{2D} \otimes \mathbb{R}^{2D}  \simeq \mathbb{R}^{4D} $, expectation values of an operator $\hat{A}$ within the Herman--Kluk setting can be reformulated as \cite{Lasser_Sattlegger:2017}
\begin{align} \label{EQ:HK_Observable}
    \langle \hat{A} \rangle_t &:=  \int_{\mathbb{R}^{4D}}  f_0(w) \Phi_t(w)  O_t[\hat{A}](w)\,dw
\end{align}
with a new variable $w=(y,z)\in\mathbb{R}^{4D}$ consisting of $y,z\in\mathbb{R}^{2D}$, each coming from a single phase space $\mathbb{R}^{2D}$. 
Here, the integrand is split into an initial-state-dependent and operator- and time-independent factor
\begin{align}\label{EQ:time_independent_part}
\begin{split}
    f_0:\mathbb{R}^{4D} \to \mathbb{C},\quad (y,z) &\mapsto (2\pi \epsilon)^{-2D}\langle \psi_0 , g_y \rangle \langle  g_z, \psi_0 \rangle,
\end{split}
\end{align}
which is identical to $f_0(y,z) = (2\pi\epsilon)^{ - D}  \overline{(\mathcal{T}\psi_0)(y)} (\mathcal{T}\psi_0)(z)$,
a time-dependent, operator- and  initial-state-independent factor 
\begin{align}\label{EQ:Phase_factor}
    \Phi_t: \mathbb{R}^{4D} \to \mathbb{C}, \quad (y,z) \mapsto \overline{R_t(y)} R_t(z) \exp{[i(S_t(z)-S_t(y))/\epsilon]},
\end{align}
and a time- and operator-dependent and initial-state-independent factor
\begin{align}\label{EQ:Time_operator_dependent}
    O_t[\hat{A}]: \mathbb{R}^{4D} \to \mathbb{C}, \quad (y,z) \mapsto\langle g_{y(t)} , \hat{A}  g_{z(t)} \rangle. 
\end{align}
To improve readability, the notation omits the $\epsilon$-dependence of the integrand. 
We introduce the following 
\begin{assumption}\label{Assumption:1}
    Let $\psi_0$ be a normalized initial state coming from  $\mathcal{S}(\mathbb{R}^D)$.
\end{assumption}
This assumption is not too restrictive as in most applications the initial state $\psi_0$ is a coherent ground state of a harmonic oscillator or a Hermite function that belong to $\mathcal{S}(\mathbb{R}^D)$. Therefore, we take \Cref{Assumption:1} from now on for granted. 

In particular, Schwartz functions are contained in the domains of operators $\hat{A}$ we will consider, and (due to the assumed smoothness and growth conditions on the potential function $V$), for all times $t\in \mathbb{R}$ the time-evolved state $\psi(t) =\hat{U}^{\rm HK}_t\psi_0$ is a Schwartz function \cite[Theorem 1]{Swart_Rousse:2008}.

We explore the characteristics of the integrand in \eqref{EQ:HK_Observable}, which will be useful in later sections.
\begin{lemma} \label{Lemma:f0_in_Schwartz}
Let $\tau>0$ be a fixed time and let $\hat{A}=\mathrm{Id}$, $\hat{q}^n_j$, $\hat{p}^n_j$, or $\hat{H}$ for  $n \in \mathbb{N}$, $j \in \{1, \dots , D\}$.

\begin{enumerate}
    \item If \Cref{Assumption:1} is true, then $f_0\in \mathcal{S}(\mathbb{R}^{4D})$.
    \item  There are $\tau$-dependent constants $c_1(\tau), c_2(\tau)>0$ such that
    \begin{align} \label{EQ:Bounds_phi}
        c_1(\tau) \leq \vert \Phi_t(w)\vert \leq c_2(\tau)
    \end{align}
    for all $t\in [0,\tau]$ and $w\in \mathbb{R}^{4D}$.
    \item 
    For all $t\in[0,\tau]$, the absolute square $\vert O_t[\hat{A}](y,z)\vert^2$ of the operator-dependent part
    grows at most polynomially as $\vert y(t) \vert, \vert z(t) \vert \to \infty$.
\end{enumerate}
\end{lemma}
\begin{proof}  
${ }$
\begin{enumerate}
    \item 
 If $\psi_0\in \mathcal{S}(\mathbb{R}^D)$, then the FBI transform satisfies $\mathcal{T}\psi_0 \in \mathcal{S}(\mathbb{R}^{2D})$ and $f_0\in \mathcal{S}(\mathbb{R}^{4D})$ as a tensor product of two Schwartz functions.
 \item Since $S_t\in\mathbb{R}$, it holds
    \begin{align}
        \vert \Phi_t(y,z) \vert  = \vert R_t(z) R_t(y)\vert.
    \end{align}
    The Herman--Kluk prefactor $R_t$ is bounded from below and from above. Hence, \eqref{EQ:Bounds_phi} holds true.
\item By the Cauchy-Schwarz inequality, it holds that
   \begin{align}
       \vert O_t[\hat{A}](y,z)\vert^2 = \vert \langle g_{y(t)}, \hat{A} g_{z(t)}  \rangle \vert^2 \leq \vert\vert \hat{A}g_{z(t)}  \vert\vert^2.
   \end{align}
\begin{enumerate}
    \item For $\hat{A}= \mathrm{Id}$ and $\hat{A}=\hat{q}_j^n$, we obtain the bound
\begin{align}
\begin{split}
    \vert O_t[\hat{A}](y,z)\vert^2 &\leq \int_{\mathbb{R}^{D}} x^{2n}_j \exp{\left(-\frac{1}{\epsilon}(x-q(t))^T \Gamma (x-q(t)) \right) } \,dx
    \\& = \textrm{Pol}(q(t)),
\end{split}
\end{align}
because the moments of a Gaussian depend polynomially on its mean. See Appendix~\ref{Appendix:Moments_of_Gaussians} for a discussion of the moments of Gaussians.
\item For a sub-quadratic potential $V$, there is a constant $c>0$ such that\begin{align}
    \vert V(x) \vert \leq \vert V(0) \vert  + \vert \nabla V(0)^T x \vert + c \sum_{j,k=1}^D \vert x_j x_k \vert.  
\end{align}
Hence, there is a polynomial $\textrm{Pol}: \mathbb{R}^D \to \mathbb{R} $ such that $ V(x)^2 \leq \textrm{Pol}(x) $ for all $x\in \mathbb{R}^D$. We obtain
\begin{align}
\begin{split}
    \vert O_t[\hat{A}](y,z)\vert^2 &\leq \int_{\mathbb{R}^{D}} V(x)^2 \exp{\left(-\frac{1}{\epsilon}(x-q(t))^T \Gamma (x-q(t)) \right) } \,dx
    \\ & \leq \int_{\mathbb{R}^{D}} \textrm{Pol}(x) \exp{\left(-\frac{1}{\epsilon}(x-q(t))^T \Gamma (x-q(t)) \right) } \,dx.
\end{split}
\end{align}
The right-hand side of the inequality depends polynomially on the mean $q(t)$. We refer again to Appendix~\ref{Appendix:Moments_of_Gaussians} for a discussion of the moments of a Gaussian. 
\item By changing to the momentum space, similar to the $n$-th order position operator $\hat{q}_j^n$, $\vert O_t[\hat{A}](y,z)\vert^2$ grows at most polynomially for $\hat{A}=\hat{p}_j^n$ as $\vert y(t) \vert \vert z(t) \vert \to \infty$. In particular, this is true for the kinetic energy operator $\hat{A}=T$.
\item Using the triangle inequality, the result follows for the Hamiltonian $\hat{H}=T+V$.
\end{enumerate}
\end{enumerate}
\end{proof}

For single-particle Gaussian-based methods, the error bounds for both wavefunctions and observables were established (see \cite{Hagedorn:1980} for the semiclassical thawed Gaussian approximation and \cite{Ohsawa:2021, Burkhard_Lasser:2023, Hagedorn:1980} for the variational Gaussian approximation). 
For both methods, the accuracy of the expectation values was higher than the wavefunction accuracy measured by the standard $L^2$ Hilbert space norm.
However, for the Herman--Kluk propagator, error bounds for observables have not yet been analyzed in detail. 

We transfer the propagator's norm accuracy to the expectation values in a straightforward manner.
\begin{proposition} \label{Proposition:Accuracy}
    Let $\psi(t) = \hat{U}_t\psi_0$ be the solution to the time-dependent Schrödinger equation \eqref{EQ:TDSE} and let \Cref{Assumption:1} hold.
    For a bounded, self-adjoint operator $\hat{A}$, the integral in \eqref{EQ:HK_Observable} approximates the expectation value $\langle \psi(t) , \hat{A} \psi(t) \rangle$ in the order of $\epsilon$. That is, for a fixed $\tau>0$ there is a constant $c(\tau)$ such that
    \begin{align}
    \sup_{t\in[0,\tau]}\vert \langle  \hat{A} \rangle_t -  \langle \psi(t), \hat{A}\psi(t) \rangle \vert \leq c(\tau) \epsilon.
\end{align}
     In the case of at most quadratic polynomial potentials $V$, the Herman--Kluk approximation becomes exact, i.\,e.,  $\langle  \hat{A} \rangle_t = \langle \psi(t), \hat{A}\psi(t) \rangle $
\end{proposition}
\begin{proof}
Let $u(t) = \hat{U}^{\rm HK}_t \psi_0$ be the Herman--Kluk wavefunction \eqref{EQ:Herman_Kluk_propagator} and recall that $\langle \hat{A} \rangle_t = \langle u(t), \hat{A}u(t) \rangle$. Using the triangle inequality, we obtain
\begin{align}
    \begin{split}
        \vert \langle \hat{A} \rangle_t - \langle \psi(t), \hat{A} \psi(t) \rangle \vert &= \vert \langle u(t) - \psi(t), \hat{A} u(t) \rangle + \langle \hat{A}\psi(t), u(t) - \psi(t) \rangle \vert 
        \\& \leq \vert\vert u(t) - \psi(t) \vert\vert \cdot \vert\vert \hat{A}\vert\vert\cdot (\vert\vert u(t)\vert\vert + \vert\vert \psi(t)\vert \vert)
        \\& \leq c(\tau) \epsilon,
    \end{split}
\end{align}
where the last step follows from the norm accuracy of the wavefunction, the boundedness of $\hat{A}$ and the norm bound of the Herman--Kluk wavefunction (see \cite[Proposition 5.4]{Lasser_Lubich:2020}).
\end{proof}

The error bound in \Cref{Proposition:Accuracy} might not be sharp. The proof presented here is naive since averaging effects of Gaussian wavepackets, used in the error bounds of single-particle Gaussian-based methods \cite{Lasser_Lubich:2020, Ohsawa:2021, Burkhard_Lasser:2023}, are not employed. It is an open question whether these beneficial averaging effects carry over to the Herman--Kluk approximation.

\subsection{Computational tasks}
The inner product \eqref{EQ:exact_exp_value} that defines the expectation value is interpreted as an integral on the double phase space $\mathbb{R}^{2D}\otimes\mathbb{R}^{2D} \simeq \mathbb{R}^{4D}$ and it can be computed with a single numerical quadrature. 
The operator $\hat{A}$ appears in the inner product \eqref{EQ:Time_operator_dependent} of two frozen Gaussians with different centres, which either simplifies the numerical quadrature or even makes it dispensable if explicit formulas are known.

To evaluate the double phase-space integral \eqref{EQ:HK_Observable} numerically, we have the following two tasks:
\begin{enumerate}
    \item The integral over $\mathbb{R}^{4D}$ has to be discretized with some quadrature rule.
    \item The integrand must be obtained from the ordinary differential equations \eqref{EQ:Traj}, \eqref{EQ:S}, and \eqref{EQ:Stab_mat}.
\end{enumerate} 
The ODEs can be solved using a symplectic integration method to preserve the symplectic structure of the underlying classical Hamiltonian system \cite{Brewer_Hulme:1997, hairer_lubich_wanner:2003, book_Hairer_Wanner:2006}. 
The computational bottleneck is the discretization of the double phase space $\mathbb{R}^{4D}$, whose dimension grows linearly with the number of dimensions $D$. 
Because of the high dimensionality and because the frozen Gaussian wavepackets can be combined with probabilistic sampling techniques, standard approaches resort to Monte-Carlo or quasi Monte-Carlo quadrature \cite{Caflisch:1998}. 
In \cite{Xie_Zhou:2021, Huang_Zhou:2022}, the authors considered Monte Carlo estimators for Herman--Kluk expectation values that scale quadratically with the number of quadrature points.
However, it is possible to construct an estimator that scales linearly with the number of quadrature points due to the interpretation $\mathbb{R}^{2D}\otimes\mathbb{R}^{2D} \simeq \mathbb{R}^{4D}$ as we will show here.

\section{Probabilistic space discretization}\label{Sec:2}
We aim to approximate the integral 
\begin{align}
\begin{split}
   \langle \hat{A} \rangle_t = \int_{\mathbb{R}^{4D}}f_0(w) \Phi_t(w) O_t[\hat{A}](w) \,dw \approx A_N(t)
\end{split}
\end{align}
by employing a probabilistic discretization technique \cite{Caflisch:1998, Liu:2004, MacKay:2005, Henning_Kersting:2022} involving $N$ quadrature points.
In the spirit of importance sampling, we introduce an arbitrary probability density $\rho > 0$ on $\mathbb{R}^{4D}$. Although the density $\rho$ may vanish on Lebesgue null sets, we omit to mention this for simplicity in the present and upcoming sections.
\begin{definition}[Monte Carlo Estimator]
For a probability density $\rho>0$ on $\mathbb{R}^{4D}$, we call
    \begin{align} \label{EQ:crude_MC_estimator}
    A_N(t) = \frac{1}{N} \sum_{j=1}^N \frac{ f_0(w_j) }{\rho(w_j)}\Phi_t(w_j)O_t[\hat{A}](w_j)
    \end{align}
the \emph{crude Monte Carlo estimator} of $\langle \hat{A} \rangle_t$ at time $t$ for which the samples $w_j\in\mathbb{R}^{4D}$ are independent and identically distributed with respect to $\rho$. 
\end{definition}
This estimator is unbiased, and the strong law of large numbers \cite[Theorem 2.4.1]{durrett:2019} provides the almost sure convergence 
\begin{align}
    A_N(t) \rightarrow \langle \hat{A} \rangle_t, \quad N\rightarrow\infty, 
\end{align}
whenever the function $f_0\Phi_t O_t[\hat{A}]: \mathbb{R}^{4D} \mapsto \mathbb{C}$ is integrable.

In particular, the convergence is independent of the chosen probability density $\rho$. 
\begin{lemma}\label{Proposition:Convergence_estimator}
    Let $n\in \mathbb{N}$, $j \in \{1,\dots, D\}$ and $\hat{A}= \mathrm{Id}$, $\hat{q}^n_j$, $\hat{p}^n_j$ or $\hat{H}$. 
    Let $\tau>0$ be a fixed time. 
    For all times $t\in [0,\tau]$, \Cref{Assumption:1} is sufficient for the almost sure convergence of the crude Monte Carlo estimator $A_N(t)$, as defined in \eqref{EQ:crude_MC_estimator}.
\end{lemma}

\begin{proof}
    Recall $f_0$, $\Phi_t$ and $O_t[\hat{A}]$ as defined in \eqref{EQ:time_independent_part}, \eqref{EQ:Phase_factor} and \eqref{EQ:Time_operator_dependent}. 
    By \Cref{Lemma:f0_in_Schwartz}, $f_0\in \mathcal{S}(\mathbb{R}^{4D})$ and there is a $\tau$-dependent constant $c(\tau)>0$ such that 
    \begin{align}
    \vert \Phi_t(y,z) \vert = \vert R_t(y) R_t(z) \vert \leq c(\tau)
    \end{align}
    for all $(y,z)\in\mathbb{R}^{4D}$ and $t\in[0,\tau]$, and
    \begin{align}
         \vert O_t[\hat{A}](y,z)\vert =\vert \langle g_{y(t)}, \hat{A}g_{z(t)} \rangle \vert
    \end{align}
    grows at most polynomially as $\vert y(t) \vert, \vert z(t) \vert \to \infty$. We deduce
    \begin{align}
        \int_{\mathbb{R}^{4D}} \vert f_0(w) \Phi_t(w) O_t[\hat{A}](w)  \vert \,dw \leq c(\tau)\int_{\mathbb{R}^{4D}} \vert f_0(w) O_t[\hat{A}](w)\vert \,dw < \infty,
    \end{align}
    since $f_0 \in \mathcal{S}(\mathbb{R}^{4D})$.
\end{proof}

The accuracy of the crude Monte Carlo estimator \eqref{EQ:crude_MC_estimator} is given by the mean squared error
\begin{align}
    \mathbb{E}\left[\vert A_N(t) - \langle \hat{A} \rangle_t \vert^2\right] = \frac{\mathbb{V}_t[\hat{A}, \psi_0, \rho]}{N},
\end{align}
whenever the variance \cite[Section 5.3]{Lasser_Sattlegger:2017}
\begin{align}\label{EQ:NotationVariance}
\begin{split}
    \mathbb{V}_t[\hat{A}, \psi_0, \rho] &= \int_{\mathbb{R}^{4D}} \left\vert \frac{f_0(w) \Phi_t(w)  O_t[\hat{A}](w)}{\rho(w)} - \langle \hat{A} \rangle_t \right\vert^2 \rho(w)\,dw
     \\&=\int_{\mathbb{R}^{4D}} \frac{1}{\rho(w)} \left\vert f_0(w) \Phi_t(w)  O_t[\hat{A}](w) \right\vert^2\, dw - \vert \langle \hat{A} \rangle_t \vert^2
 \end{split}
\end{align}
is well-defined. 
To minimize the computational cost, i.\,e., to obtain the most accurate result with a given number of samples $N$, one should choose the probability density $\rho$ such that the above variance is minimized. 
To do so, one could try to minimize the variance at the initial time $t=0$ when the auxiliary function $\Phi_t$ defined in \eqref{EQ:Phase_factor} satisfies $\Phi_0 =1$. 

In the remainder of this section, we analyse several sampling approaches in terms of their variance.

\subsection{Husimi and sqrt-Husimi sampling}\label{Sec:2.2}
There are an infinite number of possibilities for choosing the sampling density $\rho$.
In this section, we introduce the Husimi and sqrt-Husimi sampling approaches. 
The Husimi function is a commonly used sampling method in the context of the Herman--Kluk approximation \cite{Conte_Ceotto:2019, Sun_Miller:2002}, because for time-correlation functions of the form \eqref{EQ:Time_correlation_function} and for wavepacket autocorrelation functions $\langle \psi_0, \psi(t) \rangle$ the integrand is quadratic in the initial state $\psi_0$. For the time propagation of the wavefunction \eqref{EQ:Herman_Kluk_propagator} itself, the sqrt-Husimi approach was originally used in \cite{Kluk_Davis:1986}, first analyzed in \cite{Lasser_Sattlegger:2017}, and a systematic comparison with the Husimi approach is given in \cite{Kroeninger_Lasser_Vanicek:2023}.

\begin{definition}[Husimi and sqrt-Husimi densities, {\cite[Chapter 6.3]{Lasser_Lubich:2020}}]
The \emph{Husimi function} $\rho_{\mathrm{H}}$ of an initial state $\psi_0$ is defined as \cite[Chapter 6.3]{Lasser_Lubich:2020} 
\begin{align}\label{EQ:Husimi}
    \begin{split}
    \rho_{\rm H}(z)  &:= \vert  (\mathcal{T}\psi_0)(z)\vert^2
    \\&=(2\pi\epsilon)^{-D}\vert \langle g_z , \psi_0 \rangle \vert^2.
    \end{split}    
\end{align}
  Under \Cref{Assumption:1}, we call
    \begin{align}
    \begin{split}
    \rho_{\textup{sqrt-H}}(z) &:= \frac{(2\pi\epsilon)^{-D}}{\kappa_{\textup{sqrt-H}}}\vert\langle \psi_0 , g_z \rangle \vert,
    \end{split}
    \end{align}
the \emph{sqrt-Husimi density} with normalizing constant 
\begin{align}
\kappa_{\textup{sqrt-H}} :=  (2\pi\epsilon)^{-D}\int_{\mathbb{R}^{2D}} \vert\langle \psi_0 , g_{\zeta} \rangle \vert \,d\zeta.
\end{align}  
\end{definition}
By construction, $\rho_{\rm H}$ is a probability density on a single phase space $\mathbb{R}^{2D}$ for all initial data $\psi_0$ coming from $L^2(\mathbb{R}^D)$ while \Cref{Assumption:1} is sufficient for $\rho_{\textup{sqrt-H}}$ to be well-defined. 
Using the Husimi density on each phase space leads to the crude Monte Carlo estimator
\begin{align} \label{EQ:Case_H}
    A_N(t) = \frac{1}{N} \sum_{j=1}^N \frac{ f_0(w_j) }{\rho^{\rm dbl}_{\rm H}(w_j)}\Phi_t(w_j)O_t[\hat{A}](w_j) \tag{\text{Case H}}
\end{align}
with the shorthand notation $\rho^{\rm dbl}_{\rm H}:= \rho_{\rm H}\otimes \rho_{\rm H}$ and with independent $w_j\in\mathbb{R}^{4D}$ distributed with respect to $\rho^{\rm dbl}_{\rm H}$.
By the definition $f_0(w)= (2\pi\epsilon)^{-2D} \langle\psi_0, g_y \rangle \langle g_z , \psi_0 \rangle$ (see \eqref{EQ:time_independent_part}) and the Husimi function \eqref{EQ:Husimi}, the variance \eqref{EQ:NotationVariance} of the estimator \eqref{EQ:Case_H} reduces to
\begin{align} \label{EQ:Variance_Husimi}
\begin{split}
    \mathbb{V}_t[\hat{A}, \psi_0, \rho^{\rm dbl}_{\rm H}  ] &=\int_{\mathbb{R}^{4D}} \frac{\vert f_0(w) \vert^2}{\rho^{\rm dbl}_{\rm H}(w)} \left\vert \Phi_t(w)  O_t[\hat{A}](w) \right\vert^2\, dw - \vert \langle \hat{A} \rangle_t \vert^2
    \\&=(2\pi\epsilon)^{-2D} \int_{\mathbb{R}^{4D}} \vert \Phi_t(w)  O_t[\hat{A}](w) \vert^2 \,dw - \vert\langle \hat{A} \rangle_t \vert^2.
\end{split}
\end{align}

Employing $\rho_{\textup{sqrt-H}}$ instead of $\rho_{\rm H}$ on each phase space, we obtain the estimator
\begin{align} \label{EQ:Case_sqrt_H}
    A_N(t) = \frac{1}{N} \sum_{j=1}^N \frac{ f_0(w_j) }{\rho^{\rm dbl}_{\textup{sqrt-H}}(w_j)} \Phi_t(w_j)O_t[\hat{A}](w_j) \tag{\text{Case sqrt-H}}
\end{align}
where $\rho^{\rm dbl}_{\textup{sqrt-H}}:= \rho_{\textup{sqrt-H}} \otimes \rho_{\textup{sqrt-H}}$ and independent $w_j\sim \rho^{\rm dbl}_{\textup{sqrt-H}}$.
Here, the variance \eqref{EQ:NotationVariance} of the estimator \eqref{EQ:Case_sqrt_H} reduces to
\begin{align}\label{EQ:Variance_sqrtH}
\begin{split}
    \mathbb{V}_t[\hat{A}, \psi_0, \rho^{\rm dbl}_{\textup{sqrt-H}}] &=  \int_{\mathbb{R}^{4D}} \frac{\vert f_0(w) \vert^2}{\rho^{\rm dbl}_{\textup{sqrt-H}}(w)} \left\vert \Phi_t(w)  O_t[\hat{A}](w) \right\vert^2\, dw - \vert \langle \hat{A} \rangle_t \vert^2
    \\&=\kappa_{\textup{sqrt-H}}^2 \int_{\mathbb{R}^{4D}} \vert  f_0(w)\vert \vert \Phi_t(w) \vert^2 \vert O_t[\hat{A}](w)\vert^2 \,dw - \vert \langle \hat{A} \rangle_t \vert^2.
    \end{split}
\end{align}

We obtain
\begin{proposition}\label{Lemma:5}
Let $\tau>0$ be a fixed time, $n\in \mathbb{N}$, $j \in \{1,\dots, D\}$ and $\hat{A}=\mathrm{Id}$, $\hat{q}^n_j$, $\hat{p}^n_j$, $T$, $V$ or $\hat{H}$. For all times $t\in[0,\tau]$
\begin{enumerate}
    \item the variance \eqref{EQ:Variance_Husimi} of the Husimi approach is finite if and only if $\hat{A}=0$; and
    \item the variance \eqref{EQ:Variance_sqrtH} of the sqrt-Husimi approach is finite if \Cref{Assumption:1} is satisfied.
\end{enumerate}  
\end{proposition}
\begin{proof}
${ }$
\begin{enumerate}
    \item 
For the Husimi approach, it is enough to analyze the integral
\begin{align}
    \int_{\mathbb{R}^{4D}} \vert \Phi_t(w) O_t[\hat{A}](w) \vert^2 \, dw 
\end{align}
since $\vert \langle \hat{A} \rangle_t\vert^2 < \infty $. By \Cref{Lemma:f0_in_Schwartz} there are $\tau$-dependent constants $c_1, c_2>0$ such that
    \begin{align}
        c_1 \int_{\mathbb{R}^{4D}} \vert O_t[\hat{A}](w) \vert^2 \,dw \leq \int_{\mathbb{R}^{4D}} \vert \Phi_t(w) O_t[\hat{A}](w) \vert^2 \,dw \leq c_2 \int_{\mathbb{R}^{4D}} \vert O_t[\hat{A}](w) \vert^2 \,dw.
    \end{align}
Therefore, the variance $\mathbb{V}_t[\hat{A}, \psi_0, \rho_{\rm H}^{\rm dbl}]$ is finite if and only if $\vert O_t[\hat{A}](w)\vert^2$ is integrable.
Because the phase-space points $y(t)$ and $z(t)$ are solutions to Hamiltonian systems, the simplecticity of the classical flows reduces the time-dependent integral to a time-independent one, that is,
\begin{align}
\begin{split} \label{EQ:58}
    \int_{\mathbb{R}^{4D}} \vert O_t[\hat{A}](y,z) \vert^2 \,d(y,z) &= \int_{\mathbb{R}^{4D}} \vert \langle g_{y(t)}, \hat{A} g_{z(t)} \rangle \vert^2 \,d(y,z) 
    \\&= \int_{\mathbb{R}^{4D}} \vert \langle g_{y}, \hat{A} g_{z} \rangle \vert^2 \,d(y,z)
    \\& = \int_{\mathbb{R}^{2D}} \vert\vert \hat{A}g_z \vert\vert^2_{L^2(\mathbb{R}^D)} \,dz,
\end{split}
\end{align}
where we also used the isometry of the FBI transform in the last step.
\begin{enumerate}
    \item 
If $\hat{A}=0$, the integrals vanish and, therefore, the variance \eqref{EQ:Variance_Husimi}.
\item
 Let $\hat{A}=\mathrm{Id}$ or $\hat{q}_j^n$. In that case, $\vert \hat{A} g_z \vert^2$ is independent of the phase-space variable $p$ and by the non-negativity, Fubini's theorem leads to 
\begin{align}
\begin{split}
    &\int_{\mathbb{R}^{4D}} \vert O_t[\hat{A}](y,z) \vert^2 \,d(y,z) 
    \\&=\left(\frac{\det \Gamma}{\pi^D\epsilon^D} \right)^{1/2} \int_{\mathbb{R}^{D}} 1\,dp \int_{\mathbb{R}^{D}} x_j^{2n} \int_{\mathbb{R}^{D}} 
      \exp{\left(-\frac{1}{\epsilon}(x-q)^T \Gamma (x-q)  \right)}\,dq\,dx 
      \\& = \int_{\mathbb{R}^{D}} 1\,dp \int_{\mathbb{R}^{D}}  x_j^{2n} \,dx.
\end{split}
\end{align}
Hence, the variance is unbounded.
\item 
With a similar argument, $\vert O_t[\hat{A}](w)\vert^2$ is not integrable for $\hat{A}=V$ when $V\neq 0$. 
\item 
By changing to the momentum representation, $\vert O_t[\hat{A}](y,z) \vert^2$ is not integrable for $\hat{A}=\hat{p}_j^n$. In particular, it is not integrable for the kinetic energy operator $\hat{A}=T$.
\item
For the Hamiltonian $\hat{H}$, we first consider the Hessian of $g_z$, parameterized as in \eqref{EQ:Gaussian}, with respect to the spatial coordinate $x\in\mathbb{R}^D$; 
\begin{align}
    \text{Hess }g_z(x) = -\frac{g_z(x)}{\epsilon^2} \left[ \epsilon\Gamma - (\Gamma (x-q) - ip)(\Gamma (x-q) - ip)^T \right].
\end{align}
For the Hamiltonian $\hat{H}=-\epsilon^2\Delta/2 + V $ applied to the Gaussian $g_z$ we then obtain
\begin{align}
    \begin{split}
    \hat{H}g_z(x) &= V(x) g_z(x) - \frac{\epsilon^2}{2}\text{Tr}(\text{Hess }g_z(x))
    \\& = g_z(x) \left[ V(x)  + \frac{\epsilon}{2}\text{Tr}(\Gamma) - \frac{1}{2}( \Gamma (x-q) - ip )^T( \Gamma (x-q) - ip ) \right].
    \end{split}
\end{align}
Using the estimate $\vert a+ib\vert^2 \geq b^2$, $a,b \in \mathbb{R}$, we obtain the lower bound
\begin{align}
    \begin{split}
    &\int_{\mathbb{R}^{2D}}  \vert \vert \hat{H}g_z \vert \vert^2_{L^2} \,dz 
    \\&\geq \left(\frac{\det \Gamma}{\pi^D \epsilon^D} \right)^{1/2} \int_{\mathbb{R}^{3D}} (p^T \Gamma (x-q))^2 \exp{\left( -\frac{1}{\epsilon} (x-q)^T \Gamma (x-q) \right)} \,d(x,q,p)
    \\& = \left(\frac{\det \Gamma}{\pi^D \epsilon^D} \right)^{1/2} \int_{\mathbb{R}^{3D}} (p^T \Gamma x)^2 \exp{\left( -\frac{1}{\epsilon} x^T \Gamma x \right)} \,d(x,q,p)
    \end{split}
\end{align}
where the right-hand side is infinite for any potential $V$.
\end{enumerate}
This proves the statement for the variance \eqref{EQ:Variance_Husimi} of the Husimi approach.
\item
The proof of the finite variance \eqref{EQ:Variance_sqrtH} of the sqrt-Husimi approach analogously follows the proof of \Cref{Proposition:Convergence_estimator}.
\end{enumerate}
\end{proof}
We give some simple examples that illustrate \Cref{Lemma:5}.
\begin{example}\label{Ex:Variance_sqrtH}
\begin{enumerate}
    \item For the Husimi approach, we recall equality \eqref{EQ:58} and consider $\hat{A}=\textup{Id}$. The integrals reduce to 
    \begin{align}
    \begin{split}
        \int_{\mathbb{R}^{4D}} \vert \langle g_{y}, g_{z} \rangle \vert^2 \,d(y,z) &= \int_{\mathbb{R}^{4D}}  \exp{\left(-\frac{1}{2\epsilon} (z-y)^T  \begin{pmatrix}
            \Gamma & 0 \\ 0 & \Gamma^{-1}
        \end{pmatrix}  (z-y) \right)} \,d(y,z)
        \\& = (2\pi\epsilon)^{2D} \int_{\mathbb{R}^{2D}} 1 \,dy,
    \end{split}
    \end{align}
     which clearly diverges.
     \item For $z_0=(q_0,p_0)\in\mathbb{R}^{2D}$, let $\psi_0=g_{z_0}$.
Then $\rho_{\textup{sqrt-H}}$ is a Gaussian 
\begin{align}
    \rho_{\textup{sqrt-H}}(z) = (4\pi\epsilon)^{-D} \exp{\left(-\frac{1}{4\epsilon} (z-z_0)^T  \Sigma_0  (z-z_0)\right)}
\end{align}
on a single phase space $\mathbb{R}^{2D}$ with width matrix $\Sigma_0 = \mathrm{diag}(\Gamma, \Gamma^{-1})\in\mathbb{R}^{2D \times 2D}$.
\begin{enumerate}
    \item For the identity operator $\hat{A}=\mathrm{Id}$, the variance is given by
\begin{align}\label{Variance:Norm_sqrtH}
\begin{split}
    \mathbb{V}_0[\textup{Id}, g_{z_0}, \rho^{\rm dbl}_{\textup{sqrt-H}}] &= 4^D   \int_{\mathbb{R}^{4D}} \vert f_0(w) \vert \vert O_0[\textup{Id}](w)\vert^2 \,dw - \vert \langle \textup{Id} \rangle_0 \vert^2
    \\& = \left(\frac{16}{5} \right)^D -1.
\end{split} 
\end{align}
\item Considering the position operator $\hat{A} = \hat{q}_j$ for some $j\in\{1,2,\dots, D\}$ and assuming that the initial width matrix is diagonal, i.\,e., $\Gamma = \mathrm{diag}(\gamma_1, \gamma_2, \dots , \gamma_D)$, $\gamma_j>0$, it holds that
\begin{align} \label{Variance:sqrt-H_position}
   \mathbb{V}_0[\hat{q}_j, g_{z_0}, \rho^{\rm dbl}_{\textup{sqrt-H}}]
       = \frac{1}{4} \left(\frac{16}{5}\right)^D\left( 4q_{0,j}^2 + \frac{24}{5}\epsilon\gamma^{-1}_j  \right) - q_{0,j}^2,
\end{align}
where the subscript $j$ denotes the $j$-th component of a vector.
\end{enumerate}
 The derivations can be found in Appendix~\ref{Appendix:5}.
\end{enumerate}
\end{example}

The Husimi approach has the disadvantage of an unbounded variance \eqref{EQ:Variance_Husimi} of the estimator \eqref{EQ:Case_H}. In contrast, \Cref{Assumption:1} provides a sufficient condition to have a finite variance \eqref{EQ:Variance_sqrtH} for the sqrt-Husimi approach.

\Cref{Ex:Variance_sqrtH} shows an unfortunate exponential dependence of the variance of the sqrt-Husimi approach on the dimension $D$. 
Moreover, for the position operator, the exponential $D$-dependent part of the variance is multiplied by a polynomial depending on the initial position, the semiclassical parameter $\epsilon$, and the width $\Gamma$. However, it is robust in the semiclassical limit $\epsilon \to 0$. For applications with small initial values of $\gamma_j$, the included $\gamma_j^{-1}$ could lead to extensive calculations.

We conclude this section by stating that even if the Husimi approach is applicable (i.\,e., one has the existence of the first moment), one can and should directly use the sqrt-Husimi approach, as the finite variance will result in a faster convergence compared to the Husimi approach. This was also observed for the Herman--Kluk wavefunction in \cite{Kroeninger_Lasser_Vanicek:2023} and our numerical examples confirm the theoretical results for the expectation values.

\subsection{Incorporating the operator and a variational problem}\label{Sec:2.4}
In the previous section, we have shown that the Monte Carlo integrand of the sqrt-Husimi approach leads to finite variances. However, there can be an exponential dependence of the variance on the dimension $D$ leading to expensive calculations when realistic molecular systems are investigated.
This calls for improvements in the approach.
In this section, we present a generalized approach that includes the operator $\hat{A}$ in the sampling density.
\begin{theorem}[Solution to the isoperimetric problem]
Let \Cref{Assumption:1} be true and let $\hat{A}\neq 0$.
The optimal sampling density that minimizes the variance \eqref{EQ:NotationVariance} at the initial time $t=0$ is given by
\begin{align}\label{EQ:rho_opt}
    \rho_{\rm opt}(w)=\frac{1}{\kappa_{\rm opt}}\vert f_0(w) O_0[\hat{A}](w) \vert
\end{align}
with  $\kappa_{\rm opt} =\int_{\mathbb{R}^{4D}} \vert f_0(\xi) O_0[\hat{A}](\xi) \vert \, d\xi$ ensuring the normalization.
At the initial time, the variance satisfies
\begin{align}
    \mathbb{V}_0[\hat{A}, \psi_0, \rho_{\rm opt}] = \kappa_{\rm opt}^2 - \vert \langle \hat{A} \rangle_0 \vert^2.
\end{align}
\end{theorem}
\begin{proof}
First, if $\hat{A} = 0$ then $\mathbb{V}_0[\hat{A}, \psi_0, \rho]=0$ for any density $\rho$.
Therefore, let $\hat{A} \neq 0$. Then, $\rho_{\rm opt}\neq 0$ and by \Cref{Assumption:1} and \Cref{Proposition:Convergence_estimator}, the density is well-defined.
We show that $\rho_{\rm opt}$ minimizes the integral
\begin{align*}
    F(\rho) &= \int_{\mathbb{R}^{4D}} f(w,\rho(w))\,dw,
    \\ f(w,\rho(w)) &= \frac{\vert  f_0(w) O_0[\hat{A}](w) \vert^2}{\rho(w)} - \vert \langle \hat{A} \rangle_0 \vert^2\rho(w)
\end{align*}
under the constraints $\rho>0$ and $\int_{\mathbb{R}^{4D}} \rho(w)\,dw=1$.
Let $\Tilde{f}=f+\lambda g$ with $g(\rho)=\rho$ and $\lambda \in \mathbb{R}$. Then, $\rho_{\rm opt}$ is a solution to the differential equation
\begin{align}
    \begin{split}
        0 &= \frac{\partial}{\partial \rho}\Tilde{f}(\rho) 
        = -\frac{\vert  f_0 O_0[\hat{A}] \vert^2}{\rho^2} - \vert \langle \hat{A} \rangle_0 \vert^2+ \lambda
    \end{split}
\end{align}
when $\lambda = \vert \langle \hat{A} \rangle_0 \vert^2+\kappa_{\rm opt}^2 >0$.
Moreover, for $a,b \geq 0$ and $\lambda>0$ the real maps
\begin{align}
\begin{split}
    \mathbb{R}_{>0} &\to \mathbb{R}, \quad x  \mapsto \frac{a}{x}-bx, \quad \text{and}
    \\ \mathbb{R}_{>0} &\to \mathbb{R}_{>0}, \quad x  \mapsto \lambda x,
\end{split}
\end{align}
are convex. Then, by \cite[Theorem 3.16]{Troutman:1995}, $\rho_{\rm opt}$ minimizes $F(\rho)$ under the constraints $\rho>0$ and $\int_{\mathbb{R}^{4D}}\rho(w)\,dw=1$.
The variance $\mathbb{V}_0[\hat{A}, \psi_0, \rho_{\rm opt}]$ immediately follows  by inserting $\rho_{\rm opt}$ into \eqref{EQ:NotationVariance}.
\end{proof}
Now, with $\rho_{\rm opt}$ defined in \eqref{EQ:rho_opt}, the crude Monte Carlo estimator for $\langle \hat{A}\rangle_t$ reads
\begin{align} \label{EQ:Case_opt}
    A_N(t) = \frac{1}{N} \sum_{j=1}^N \frac{ f_0(w_j) }{\rho_{\rm opt}(w_j)}\Phi_t(w_j)O_t[\hat{A}](w_j) \tag{\text{Case opt}}
\end{align}
with independent samples $w_j\sim \rho_{\rm opt}$. 
Sampling from a general $\rho_{\rm opt}$ can be done using a Metropolis--Hastings algorithm such as the Hamiltonian Monte Carlo algorithm described in Appendix~\ref{Appendix:2}.
The time dependence of the variance of the estimator \eqref{EQ:Case_opt} is given by
\begin{align}\label{EQ:Variance_opt}
    \mathbb{V}_t[\hat{A}, \psi_0, \rho_{\rm opt}] = \kappa_{\rm opt} \int_{\mathbb{R}^{4D}} \vert f_0(w) \vert \vert \Phi_t(w) \vert^2 \frac{\vert O_t[\hat{A}](w) \vert^2}{ \vert O_0[\hat{A}](w) \vert} \,dw - \vert \langle \hat{A} \rangle_t \vert^2.
\end{align}
Compared with the variances \eqref{EQ:Variance_Husimi} and \eqref{EQ:Variance_sqrtH} of the previous two approaches, the variance \eqref{EQ:Variance_opt} now includes $\vert O_0[\hat{A}]\vert $ in the denominator, resulting in a reduced value.
Because of the normalization constant $\kappa_{\rm opt}$, the variance may depend on the dimension $D$ and various parameters such as the initial position, momentum, and width.
In a simple example, we compare the variance \eqref{EQ:Variance_opt} of the optimal approach with that of the sqrt-Husimi approach \eqref{EQ:Variance_sqrtH}.
\begin{example}\label{Ex:Variance_opt}
Let $\psi_0 = g_{z_0}$ be a Gaussian wavepacket with an initial phase-space point $z_0=(q_0,p_0)\in\mathbb{R}^{2D}$ and let $\hat{A}=\mathrm{Id}$. Then $\rho_{\rm opt}$ is a Gaussian 
\begin{align}
    \rho_{\rm opt}(y,z) = (2\pi\epsilon)^{2D} \left(\frac{3}{4}\right)^D \exp{\left(-\frac{1}{4\epsilon} \begin{pmatrix}
        y-z_0 \\ z- z_0
    \end{pmatrix}^T \begin{pmatrix}
        2\Sigma_0 & -\Sigma_0 \\ - \Sigma_0 & 2\Sigma_0
    \end{pmatrix} \begin{pmatrix}
        y-z_0 \\ z- z_0
    \end{pmatrix}\right)}
\end{align}
on the double phase space $\mathbb{R}^{4D}$ with $\Sigma_0=\mathrm{diag}(\Gamma, \Gamma^{-1})\in \mathbb{R}^{2D \times 2D}$, $\kappa_{\rm opt} = (4/3)^D$ and 
\begin{align}\label{Variance:Norm_opt}
  \mathbb{V}_0[\textup{Id}, g_{z_0}, \rho_{\rm opt}]= \left(\frac{4}{3}\right)^{2D}-1 =\left(\frac{16}{9}\right)^D-1 < 2^D-1.
\end{align}
The sqrt-Husimi approach provided $\mathbb{V}_0[\textup{Id}, g_{z_0}, \rho_{\textup{sqrt-H}}^{dbl} ] = \left(16/5 \right)^D-1 > 3^D-1$.
\end{example}
In conclusion, for a given number of quadrature points, the most accurate approximation of the integral \eqref{EQ:exact_exp_value} can be obtained by including the operator $\hat{A}$ in the sampling density. The density $\rho_{\rm opt}$ is the optimal choice that minimizes the variance \eqref{EQ:NotationVariance} at the initial time $t=0$. However, for most choices of $\hat{A}$, the density $\rho_{\rm opt}$ is only known up to a constant and the crude Monte Carlo estimator \eqref{EQ:crude_MC_estimator} will not be applicable as the next example shows.
\begin{example}
    The normalization constant of $\rho_{\rm opt}$ is given by
    \begin{align}
        \kappa_{\rm opt} &= (2\pi\epsilon)^{-2D}\int_{\mathbb{R}^{4D}} \vert \langle \psi_0, g_z \rangle \langle g_y, \psi_0 \rangle \vert \vert \langle g_y, \hat{A}g_z \rangle \vert \,d(y,z).
    \end{align}
    In the case of $\hat{A}=V$ with a general potential $V$ and a more complex initial state $\psi_0$, the explicit value of $\kappa_{\rm opt}$ may not be derivable through a straightforward calculation.
    Therefore, $\rho_{\rm opt}$ is only known up to a constant.
\end{example}

\section{Weighted Importance sampling}\label{Sec:3}
To use the optimal sampling approach, we extend the initial crude Monte Carlo estimator \eqref{EQ:crude_MC_estimator} to a self-normalizing version based on the general idea of including an additional weight  (see, for example, \cite[Section 2.5.3]{Liu:2004} or \cite[Section 9.7.4]{Kroese:2011}).
We will first motivate the estimator in a general setting, then analyze some of its properties, and afterwards compare different weighting approaches.

Let $\rho_1, \rho_2 > 0$ be two densities on the double phase space $\mathbb{R}^{4D}$ and define the weight $W:=\rho_1/\rho_2 > 0$. 
The expectation value of an operator $\hat{A}$ within the Herman--Kluk approximation defined in \eqref{EQ:HK_Observable} can be written as the quotient
\begin{align}\label{EQ:WIS_motivation}
\begin{split}
    \langle \hat{A} \rangle_t &= \int_{\mathbb{R}^{4D}} f_0(w) \Phi_t(w) O_t[\hat{A}](w) \,dw
    \\& = \left.\int_{\mathbb{R}^{4D}} \frac{f_0(w)}{\rho_1(w)} \Phi_t(w) O_t[\hat{A}](w) W(w) \rho_2(w) \,dw\right/\int_{\mathbb{R}^{4D}} W(w)\rho_2(w) \,dw,
\end{split}
\end{align}
since
\begin{align}
    \int_{\mathbb{R}^{4D}} W(w) \rho_2(w) \,dw = \int_{\mathbb{R}^{4D}} \rho_1(w) \,dw = 1.
\end{align}
By linearity, the weight $W$ can be replaced by $\kappa W$ for some constant $\kappa \neq 0$.

\begin{definition}[Weighted importance sampling estimator]
Let $\rho_1>0$ be a known density and let $\rho_2>0$ be the desired sampling density, both acting on the double phase space $\mathbb{R}^{4D}$.
Let $w_1,w_2, \dots , w_N\in\mathbb{R}^{4D}$ be independent and identically distributed with respect to $\rho_2$. 
We define the \emph{weighted importance sampling estimator} of $\langle\hat{A}\rangle_t$ as
\begin{align} \label{EQ:weighted_MC_estimator}
\begin{split}
    A_N^{\rm W}(t)= \frac{\sum_{j=1}^N g_t(w_j)W(w_j)}{\sum_{j=1}^N W(w_j)}
\end{split}
\end{align}
with the weight $W=\rho_1/\rho_{2} > 0$ and the function
\begin{align}
    g_t(w)= \frac{ f_0(w)}{\rho_1(w)}\Phi_t(w) O_t[\hat{A}](w) .
\end{align} 
\end{definition}

One can directly see, in the case of $\rho_1=\rho_{2}$, the weighted importance sampling estimator \eqref{EQ:weighted_MC_estimator} reduces to the crude Monte Carlo estimator \eqref{EQ:crude_MC_estimator}.
Moreover, the features of the estimator \eqref{EQ:weighted_MC_estimator} are provided by 
\begin{theorem}[Consistency, bias and variance] \label{Theorem:Properties_WIS}
Let $w_1, \dots , w_N\in\mathbb{R}^{4D}$ be independent and identically distributed with respect to $\rho_2$ and let $\tau>0$ be a fixed time.
The weighted importance sampling estimator 
    \begin{align}
    \begin{split}
    A_N^{\rm W}(t)= \frac{\sum_{j=1}^N g_t(w_j)W(w_j)}{\sum_{j=1}^N W(w_j)}
    \end{split}
    \end{align}
    as defined in \eqref{EQ:weighted_MC_estimator} has the following properties:
    \begin{enumerate}
    \item The estimator $A_N^{\rm W}(t)$
    converges in mean to $\langle \hat{A} \rangle_t$ for all times $t\in[0, \tau]$ as $N\to \infty$, that is,
    \begin{align}
        \mathbb{E}[\vert A_N^{\rm W}(t) - \langle \hat{A} \rangle_t \vert] \xrightarrow{N\to \infty} 0. 
    \end{align}
    \item The estimator $A_N^{\rm W}$ is biased 
    and the bias vanishes for $N\rightarrow \infty$.
    \item If $\mathbb{E}[\vert g_tW\vert^2]$ and $\mathbb{E}[\vert W\vert^2]$ are finite, then the bias, the variance, and the mean squared error of the estimator $A_N^{\rm W}$ are in the order of $\mathcal{O}(N^{-1})$.
\end{enumerate}    
\end{theorem}

\subsection{Proof of \Cref{Theorem:Properties_WIS}}
In the following, we provide the proof of \Cref{Theorem:Properties_WIS}. We first examine the convergence, that is guaranteed by the finite time interval $[0,\tau]$, of the averages that appear in the numerator and denominator of \eqref{EQ:weighted_MC_estimator}. Then, we employ Taylor expansions of the estimator $A_N^{\rm W}$ to assess the bounds for the convergence error, bias, variance, and mean squared error which depend on the individual sample averages.
We emphasize that \Cref{Assumption:1} is taken for granted.

We start with two helpful results:

\begin{lemma}\label{Lemma:Individual_sample_averages}
    Let $\tau>0$ be a fixed time, let $\rho_1$, $\rho_2>0$ be arbitrary densities on the double phase space $\mathbb{R}^{4D}$ and let $w_1,\dots ,w_N\in \mathbb{R}^{4D}$ be samples independent and identically distributed with respect to $\rho_2$.
    The individual sample averages
    \begin{align}
        G_N(t):= \frac{1}{N}\sum_{j=1}^N g_t(w_j)W(w_j)
    \end{align}
    and 
    \begin{align}
        W_N:= \frac{1}{N}\sum_{j=1}^N W(w_j)
    \end{align}
    converge almost surely to the means $\mathbb{E}[G_N(t)]=\langle \hat{A} \rangle_t$ and $\mathbb{E}[W_N]=1$ for all times $t\in [0,\tau]$ as $N\to \infty$. Moreover, they also converge in $L^1$.
\end{lemma}
\begin{proof}
    Because the samples $w_1,\dots ,w_N\in \mathbb{R}^{4D}$ are independent and identically distributed with respect to $\rho_2$, it holds that
    \begin{align}
        \mathbb{E}[G_N(t)]&=\frac{1}{N}\mathbb{E}[\sum_{j=1}^N g_t(w_j)W(w_j)]
        =\int_{\mathbb{R}^{4D}} g_t(w)W(w) \rho_2(w) \,dw
        = \langle \hat{A} \rangle_t
    \end{align}
    as well as
    \begin{align}
        \mathbb{E}[W_N]=\frac{1}{N}\mathbb{E}[\sum_{j=1}^N W(w_j)]= \int_{\mathbb{R}^{4D}} W(w) \rho_2(w)\,dw = 1.
    \end{align}
    Therefore, $W_N$ and $G_N$ are unbiased estimators with means $\mathbb{E}[W_N]=1$ and $\mathbb{E}[G_N(t)]=\langle \hat{A} \rangle_t$.
    Moreover, since $\rho_1$ and $\rho_2$ are two densities on the double phase space $\mathbb{R}^{4D}$, we obtain
    \begin{align}
        \mathbb{E}[\vert W \vert] &= \int_{\mathbb{R}^{4D}} \vert W(w) \vert \rho_2(w) \,dw 
        = \int_{\mathbb{R}^{4D}} \frac{\rho_1(w)}{\rho_2(w)} \rho_2(w) \,dw=1
    \end{align}
     and, by \Cref{Assumption:1} and \Cref{Proposition:Convergence_estimator}, it follows 
    \begin{align}
        \mathbb{E}[\vert g_tW \vert] = \int_{\mathbb{R}^{4D}} \vert g_t(w) W(w) \vert \rho_2(w)\,dw = \int_{\mathbb{R}^{4D}} \vert f_0(w) \Phi_t(w) O_t[\hat{A}](w) \vert\,dw < \infty.
    \end{align}
    Hence, applying the strong law of large numbers shows almost sure convergence of the estimators $G_N(t)$ and $W_N$ for all times $t\in [0, \tau]$ as $N\to \infty$. The convergence in $L^1$ also follows from the strong law of large numbers \cite[Remark 3, Section 4.4]{Shiryaev:2021}.
\end{proof}
To assess the bounds for the convergence error, bias, variance, and mean squared error we use
\begin{lemma}[Taylor expansions of $A_N^{\rm W}$] \label{Lemma:TaylorExpansion}
    Let the samples $w_1, \dots , w_N\in\mathbb{R}^{4D}$ be independent and identically distributed with respect to $\rho_2$.
    Then, the zeroth- and first-order Taylor expansions of $A_N^{\rm W}$, as a function from $\mathbb{C} \times \mathbb{R}$ to $\mathbb{C}$, around the vector $\mu=(\langle \hat{A} \rangle_t, 1 ) \in \mathbb{C}\times \mathbb{R}_{>0}$ are given by
    \begin{align}
        \begin{split}
            A_N^{\rm W}&=\langle \hat{A} \rangle_t + R_0(G_N(t),W_N)
        \end{split}
    \end{align}
    and
    \begin{align}
     A_N^{\rm W}&=\langle \hat{A} \rangle_t + (G_N(t)-\langle \hat{A} \rangle_t) - \langle \hat{A} \rangle_t(W_N-1)   + \frac{1}{2}R_1(G_N(t),W_N),
    \end{align}
    with the complex-valued remainders $R_0$ and $R_1$ acting on $\mathbb{C} \times \mathbb{R}$.
Moreover, with probability one, there is $N_0\in \mathbb{N}$ and a constant $c>0$ depending on $N_0$ such that 
    \begin{align}
    \begin{split}
        &\vert R_0(G_N(t), W_N) \vert \leq c(N_0)(\vert G_N(t)-\langle \hat{A} \rangle_t \vert + \vert W_N-1\vert ),
        \\& \textrm{and}
        \\& \vert R_1(G_N(t), W_N) \vert \leq c(N_0)(\vert G_N(t)-\langle \hat{A} \rangle_t \vert \vert W_N-1\vert + \vert W_N-1\vert^2 )
    \end{split}
    \end{align}
    for all $N\geq N_0$.
\end{lemma}

\begin{proof}
   Consider the function 
   \begin{align}
    f: \mathbb{C}\times \mathbb{R}_{>0} \rightarrow \mathbb{C}, \quad (a, b) \mapsto \frac{a}{b}.
    \end{align}
    First, by the definition of $A_N^{\rm W}$, one can see that $f(G_N(t), W_N)=A_N^{\rm W}$. Then, the Taylor expansions of $A_N^{\rm W}$ follow directly from Taylor's theorem applied to $f$ with the remainders
    \begin{align}
    \begin{split}
     R_0(G_N(t), W_N) &=  (G_N(t)-\langle \hat{A} \rangle_t)\int_0^1 \frac{1}{1+s(W_N-1)} \,ds
    \\&- (W_N-1)\int_0^1 \frac{\langle \hat{A} \rangle_t + s(G_N(t) -\langle \hat{A} \rangle_t )}{[1+s(W_N-1)]^2}\,ds,
    \end{split}
\end{align}
and
\begin{align}
\begin{split}
    R_1(G_N(t),W_N) &= 2(W_N-1)^2\int_0^1 (1-s)\frac{\langle \hat{A} \rangle_t + s(G_N(t) -\langle \hat{A} \rangle_t )}{[1+s(W_N-1)]^3}\,ds 
     \\& - (G_N(t)-\langle \hat{A} \rangle_t)(W_N-1) \int_0^1 (1-s)\frac{1}{[1+s(W_N-1)]^2}  \,ds.
\end{split}
\end{align}

   Second, by \Cref{Lemma:Individual_sample_averages}, $G_N(t)$ and $W_N$ converge almost surely to $\langle\hat{A}\rangle_t$ and $1$ as $N\to \infty$. 
    With probability one, there is $0<\delta<1$ and $N_0 \in \mathbb{N}$ such that 
    \begin{align}
        (G_N(t),W_N) \in B_{\delta}(\langle \hat{A} \rangle_t) \times [1-\delta, 1+\delta]
    \end{align}
    for all $N> N_0$ with the closed ball $B_{\delta}(\langle \hat{A} \rangle_t) = \{a\in \mathbb{C}: \vert a-\langle \hat{A} \rangle_t \vert \leq \delta \}$ . 
    Therefore, by the extreme value theorem, there is a constant $c_0>0$ such that
    \begin{align}
    \begin{split}
        \vert R_0(G_N(t),W_N) \vert &\leq  \vert G_N(t)-\langle \hat{A} \rangle_t  \vert \int_0^1  \vert \partial_af(\mu+s((G_N(t),W_N)-\mu)) \vert  \,ds
    \\&+  \vert W_N-1 \vert\int_0^1  \vert\partial_bf(\mu+s((G_N(t),W_N)-\mu)) \vert \,ds
    \\& \leq c_0 (\vert G_N(t)- \langle \hat{A} \rangle_t \vert + \vert W_N- 1 \vert )
    \end{split}
    \end{align}
    for all $N$ large enough. 
    Similarly, one obtains the bound 
    \begin{align}
           \vert R_1(G_N(t), W_N) \vert \leq c_1(\vert G_N(t)-\langle \hat{A} \rangle_t \vert \vert W_N-1\vert + \vert W_N-1\vert^2 )
    \end{align}
    with a constant $c_1>0$.
    The choice of $c=\max\{c_0,c_1\}$ proves the lemma.
\end{proof}

Combining \Cref{Lemma:Individual_sample_averages} and \Cref{Lemma:TaylorExpansion} we obtain the proof of \Cref{Theorem:Properties_WIS}.

\begin{proof}[Proof of \Cref{Theorem:Properties_WIS}]
${ }$ \newline
    1. \Cref{Lemma:TaylorExpansion} implies that
    \begin{align}
    \begin{split}
        \mathbb{E}[\vert A_N^{\rm W} - \langle \hat{A} \rangle_t \vert] &= \mathbb{E}[\vert \langle \hat{A} \rangle_t + R_0(G_N(t),W_N) - \langle \hat{A} \rangle_t \vert]
        \\&=\mathbb{E}[\vert R_0(G_N(t), W_N) \vert ]
        \\& \leq c (\mathbb{E}[\vert G_N(t) - \langle \hat{A} \rangle_t \vert] +\mathbb{E}[\vert W_N - 1 \vert] )
    \end{split}
    \end{align}
    for some constant $c>0$ and for $N$ sufficiently large. 
    By \Cref{Lemma:Individual_sample_averages}, the right-hand side vanishes as $N\to \infty$. Therefore, the estimator is consistent.
\\
2. By \Cref{Lemma:TaylorExpansion}, the bias of the estimator satisfies the equality
\begin{align}
\begin{split}
    \textrm{Bias}(A_N^{\rm W}(t)) 
     = \vert \mathbb{E}[R_0(G_N(t), W_N)] \vert.  
\end{split}
\end{align}
For $N$ large enough, the remainder can be bounded by \Cref{Lemma:TaylorExpansion} and we obtain
\begin{align}
    \textrm{Bias}(A_N^{\rm W}(t)) \leq c (\mathbb{E}[\vert G_N(t) - \langle \hat{A} \rangle_t \vert] +\mathbb{E}[\vert W_N - 1 \vert] ),
\end{align}
for some constant $c>0$. The bias vanishes as $N\to \infty$ by \Cref{Lemma:Individual_sample_averages}.
\\
3. If $\mathbb{E}[\vert g_tW \vert^2], \mathbb{E}[\vert W \vert^2]< \infty$, the bias obeys
\begin{align}
    \begin{split}
        &\textrm{Bias}(A_N^{\rm W}(t)) 
        \\&= \vert \mathbb{E}[ \langle \hat{A} \rangle_t + (G_N(t)-\langle \hat{A} \rangle_t) - \langle \hat{A} \rangle_t(W_N-1)   + \frac{1}{2}R_1(G_N(t),W_N)  ] - \langle \hat{A} \rangle_t \vert 
        \\& = \frac{1}{2}\vert \mathbb{E}[R_1(G_N(t), W_N)]  \vert,
    \end{split}
\end{align}
by \Cref{Lemma:Individual_sample_averages} and \Cref{Lemma:TaylorExpansion}. 
Moreover, by \Cref{Lemma:TaylorExpansion}, the remainder can be bounded for $N$ large enough and the Cauchy-Schwarz inequality leads to
\begin{align}
\begin{split}
    \textrm{Bias}(A_N^{\rm W}(t)) &\leq \frac{c}{2} \mathbb{E}[\vert G_N(t)-\langle \hat{A} \rangle_t \vert \vert W_N-1\vert + \vert W_N-1\vert^2] 
    \\&\leq \frac{c}{2N} (\textup{Var}[W] + \sqrt{\textup{Var}[W]\textup{Var}[g_tW]}) 
    \end{split}
\end{align}
for some constant $c>0$.
Assumptions $\mathbb{E}[\vert g_tW \vert^2]$ and $\mathbb{E}[\vert W \vert^2]< \infty$ imply that $\textrm{Bias}(A_N^{\rm W}(t))=\mathcal{O}(N^{-1})$. Using  similar arguments, we find that the variance of the estimator satisfies
\begin{align}
\begin{split}
\textup{Var}[A_N^{\rm W}(t)] &= \textup{Var}[R_0(G_N(t), W_N)]  
\\&\leq \mathbb{E}[\vert R_0(G_N(t), W_N) \vert^2]
\\& \leq c^2 \mathbb{E}[(\vert G_N(t)-\langle \hat{A} \rangle_t \vert + \vert W_N-1 \vert )^2 ]
\\ & = c^2 \mathbb{E}[\vert G_N(t)-\langle \hat{A} \rangle_t \vert^2 + \vert W_N-1 \vert ^2 + 2 \vert G_N(t)- \langle \hat{A} \rangle_t \vert \vert W_N -1 \vert]
\end{split}
\end{align}
for $N$ large enough. Therefore, $\textrm{Var}[A_N^{\rm W}(t)] = \mathcal{O}(N^{-1})$.
Finally, the mean squared error is given by
   \begin{align}
        \begin{split}
            \mathbb{E}[\vert A_N^{\rm W}(t) - \langle \hat{A} \rangle_t  \vert^2] & = \mathbb{E}[\vert A_N^{\rm W}(t) - \mathbb{E}[A_N^{\rm W}(t)] - \langle \hat{A}  \rangle_t +\mathbb{E}[A_N^{\rm W}(t)]\vert^2]
            \\& = \text{Var}[A_N^{\rm W}(t)] +  \text{Bias}(A_N^{\rm W}(t)) ^2
            \\& = \mathcal{O}(N^{-1}),
        \end{split}
    \end{align}
    as the variance and bias are of order $\mathcal{O}(N^{-1})$.
\end{proof}

From a computational point of view, $\rho_2$ might only be known up to a constant $\kappa\neq 0$, that is, $\rho_2/\kappa$ is known. Assume that the samples $w_j\sim \rho_2$ can still be generated. 
By the linearity, one can then replace $W=\rho_1/\rho_2$ in the estimators $G_N(t)$, $W_N$ and $A_N^{\rm W}$ with $W=\kappa\rho_1/\rho_2$. 
Because this is a constant factor, the properties of \Cref{Theorem:Properties_WIS} remain the same.   
Without loss of generality and for the sake of readability, we omit to mention any such constant $\kappa$ here and in the following sections.

\begin{remark}
    The analysis of the weighted importance sampling estimator \eqref{EQ:weighted_MC_estimator} is a generalization of the analysis of the crude Monte Carlo estimator \eqref{EQ:crude_MC_estimator}. 
    If one sets $\rho_1=\rho_{2}$, all stated formulas for the bias, variance, and mean squared error reduce to the corresponding formulas for the crude Monte Carlo estimator \eqref{EQ:crude_MC_estimator}.
\end{remark}
\begin{remark}
    Assuming the existence of fourth moments $\mathbb{E}[\vert g_tW\vert^4]$ and $\mathbb{E}[\vert W\vert^4]< \infty$, one can prove the estimate
    \begin{align}\label{EQ:variance_weighted_is}
    \textup{Var}[A_N^{\rm W}(t)]= \frac{\textup{Var}[g_tW-\langle \hat{A} \rangle_t W]}{N} + \mathcal{O}(N^{-2})
    \end{align}
    by using a second-order Taylor expansion of $A_N^{\rm W}$ around the vector $\mu=(\langle \hat{A} \rangle_t, 1)$.
\end{remark}

\subsection{Choice of the weight $W$}
By \Cref{Theorem:Properties_WIS}, for any choice of the weight $W$, the estimator is consistent and the bias vanishes for $N\rightarrow\infty$. 
Moreover, if $\mathbb{E}[\vert g_tW\vert^2],\mathbb{E}[\vert W\vert^2]< \infty$, the bias, variance, and mean squared error of the estimator are in the order of $\mathcal{O}(N^{-1})$. 
In the following, we briefly discuss various choices for the weight $W$. 

First, we provide an example in which an explicit formula for the variance \eqref{EQ:variance_weighted_is} of the weighted importance sampling estimator \eqref{EQ:weighted_MC_estimator} is known.
\begin{example}\label{EX:3.6}
    Consider $\hat{A}=\mathrm{Id}$ at initial time $t=0$ with an initial Gaussian wavepacket $\psi_0 = g_{z_0}$ for some initial phase-space point $z_0\in\mathbb{R}^{2D}$. For $\rho_1=  \rho^{\rm dbl}_{\textup{H}}$ and $\rho_2 = \rho^{\rm dbl}_{\textup{sqrt-H}}$, finite fourth moments of $g_0W$ and $W$ exist.
    Then,
    \begin{align}
    \mathrm{Var}[g_0W-W] = \mathbb{E}[\vert g_0W\vert^2 + \vert W \vert^2 - 2 \mathrm{Re}(g_0W^2)]+ \mathcal{O}(N^{-2})
    \end{align}
     and the first two terms are given by
    \begin{align}
    \begin{split}
    \mathbb{E}[\vert g_0W\vert^2] = \int_{\mathbb{R}^{4D}} \frac{\vert f_0(w) O_0[\mathrm{Id}](w) \vert^2}{\rho_2(w)} &= 4^D \int_{\mathbb{R}^{4D}} \vert f_0(w) \vert \vert O_0[\mathrm{Id}](w) \vert^2\,dw  = \left( \frac{16}{5} \right)^D,
    \\ \textrm{and } \quad \mathbb{E}[\vert W\vert^2] &= \int_{\mathbb{R}^{4D}} \frac{\rho_1(w)^2}{\rho_2(w)} \,dw = \left(\frac{16}{9}\right)^D.
    \end{split}
    \end{align}
    Finally, by applying several Fourier transforms of Gaussians, we find that
    \begin{align}
     \mathbb{E}[g_0W^2] = \int_{\mathbb{R}^{4D}} f_0(w) O_0[\mathrm{Id}](w) \frac{\rho_1(w)}{\rho_2(w)}   =\left(\frac{16}{9}\right)^D
    \end{align}
    and we obtain the variance of the estimator
    \begin{align}\label{Ex:Eq_variance}
        \textup{Var}[A_N^{\rm W}(0)] = \left(\left(16/5 \right)^D - \left(16/9 \right)^D\right)/N + \mathcal{O}(N^{-2}).
    \end{align}
    In the case of $D=10$, the term $(16/5)^D-(16/9)^D$ in the variance \eqref{Ex:Eq_variance} is approximately $1.1\times 10^6$.
\end{example}

There is also an example with an unbounded second moment of the weight $W$ and, therefore, not all properties stated in \Cref{Theorem:Properties_WIS} will be applicable.
\begin{example}[Diverging second moment]\label{Ex:DivergingWeight}
Let  $\hat{A}=\mathrm{Id}$, $\rho_1 =\rho^{\rm dbl}_{\textup{sqrt-H}}$ and $\rho_2 = \rho_{\rm opt}$ with initial Gaussian wavepacket $g_{z_0}$ and an initial phase-space point $z_0\in\mathbb{R}^{2D}$. Recall that $\rho_{\rm opt}$ is proportional to
\begin{align}
    \rho_{\rm opt}(w) \propto \vert  \vert f_0(w) \vert \vert \langle g_y , g_z \rangle \vert.
\end{align}
Then,
\begin{align}\label{EQ:Ex_5_2}
    \frac{\vert \rho^{\rm dbl}_{\textup{sqrt-H}}  \vert^2}{\vert\rho_{\rm opt}\vert} \propto \exp\left[-\frac{1}{8\hbar} \begin{pmatrix}
        z-z_0 \\y-z_0
    \end{pmatrix}^T \begin{pmatrix}
        0 & \Gamma \\ \Gamma & 0
    \end{pmatrix} \begin{pmatrix}
        z-z_0 \\y-z_0
    \end{pmatrix} \right].
\end{align}
The width matrix in \eqref{EQ:Ex_5_2} is not positive-definite. Hence, the second moment $\mathbb{E}[\vert W \vert^2]$ of the weight $W=\rho_1/\rho_2$ does not exist. 
\end{example}

In \Cref{NumericalExample_Weight}, we consider numerical examples with several choices of the weight $W$.
It turns out that, the choice $\rho_1 = \rho^{\rm dbl}_{\rm H}$ and $\rho_{2}=\rho_{\rm opt}$ has the best performance.
This suggests the following weighted importance sampling estimator
\begin{align}\label{Estimator_WIS_opt}
    A_N^{\rm W}(t) = \frac{\sum_{j=1}^N f_0(w_j) \Phi_t(w_j) O_t[\hat{A}](w_j) / \rho_{\rm opt}(w_j)}{\sum_{j=1}^N \rho^{\rm dbl}_{\rm H}(w_j)/\rho_{\rm opt}(w_j)}, \tag{Case WIS opt}
\end{align}
with independent samples $w_j \sim \rho_{\rm opt}$.

\subsection{Algorithm}\label{Sec:4}
\begin{algorithm}
\caption{Evaluation of expectation values}\label{Alg:1}
\begin{algorithmic}
\State{Given an initial state $\psi_0\in \mathcal{S}(\mathbb{R}^D),$ a sampling density $\rho$ and an operator $\hat{A}$, evaluate $\langle \hat{A} \rangle_t$ as follows:
\begin{enumerate}
        \item \label{Alg1:step_1} Sample double phase-space coordinates $w_1, \dots , w_N\in\mathbb{R}^{4D},$ $w_j=(z_j,y_j)$, distributed with respect to $ \rho$. 
        \item \label{Alg1:step_2} For all $j\in\{ 1, 2, \dots , N\}$:
        \begin{enumerate}
            \item Set initial values $z(0) = z_j, y(0)= y_j,$ $M(0,z_j) = M(0,y_j) = \mathrm{Id}_{2D}$ and $S_0(z_j) = S_0(y_j)= 0$.
            \item Compute approximate solutions to  \eqref{EQ:Traj}, \eqref{EQ:S}, and \eqref{EQ:Stab_mat} for each phase-space point $z_j$ and $y_j$ up to time $t$ using a symplectic integrator.
            \item Compute the Herman--Kluk prefactors $R_t(y_j)$ and $R_t(z_j)$ from $M(t,y_j)$ and $M(t,z_j)$ while choosing the correct branch of the complex square root to ensure continuity as a function of time. 
            \item Evaluate $f_0(w_j)$, $\Phi_t(w_j)$ and $O_t[\hat{A}](w_j)$ as defined in \eqref{EQ:time_independent_part}, \eqref{EQ:Phase_factor} and \eqref{EQ:Time_operator_dependent}.
        \end{enumerate}
        \item \label{Alg1:step_3} Evaluate $\langle \hat{A} \rangle_t$ either by the crude Monte Carlo estimator
        \begin{align}
            A_N(t) = \frac{1}{N}\sum_{j=1}^N \frac{f_0(w_j)}{\rho(w_j)} \Phi_t(w_j) O_t[\hat{A}](w_j) 
        \end{align}
        or by the weighted importance sampling estimator
        \begin{align}
            A_N^{\rm W}(t) = \frac{\sum_{j=1}^N \Phi_t(w_j)f_0(w_j) O_t[\hat{A}](w_j)/\rho(w_j)}{\sum_{j=1}^N \rho^{\rm dbl}_{\rm H}(w_j)/\rho(w_j)},
        \end{align}
        where $\rho_{\rm H}^{\rm dbl} = \rho_{\rm H} \otimes \rho_{\rm H} $.
    \end{enumerate}
    }
\end{algorithmic}
\end{algorithm}
We have analysed several sampling strategies for Herman--Kluk expectation values based on the Husimi, sqrt-Husimi (Sec.~\ref{Sec:2.2}), and the optimal approach  (Sec.~\ref{Sec:2.4}) that either use the crude Monte Carlo  \eqref{EQ:crude_MC_estimator} or the weighted importance sampling estimator \eqref{EQ:weighted_MC_estimator}. 
Hence, we propose the following general \Cref{Alg:1} to compute the expectation value of an operator $\hat{A}$ at a given time $t$ within the Herman--Kluk approximation.

For the Markov-Chain-Monte-Carlo algorithm in \Cref{Alg:1}, we use the Hamiltonian Monte Carlo algorithm described in Appendix~\ref{Appendix:2}.
With a fictitious, random momentum variable and a carefully chosen Hamiltonian, it uses a symplectic, time-reversible integrator that propagates an initial double phase-space variable $w_0\in\mathbb{R}^{4D}$ according to Hamiltonian's equations of motion and accepts or rejects generated points $w_j \in\mathbb{R}^{4D}$. 

For the Hamiltonian Monte Carlo Algorithm \ref{Alg:HMC} (see Appendix~\ref{Appendix:2}), we use a second-order St\o rmer-Verlet integrator for the propagation. 
Recently analyzed multi-stage integration schemes \cite{Blanes_Casas_Sanz_Serna:2014, Nagar_Pendas_Sanz_Serna_Akhmatskaya:2023, Casas_Sanz_Serna_Shaw:2022, Blanes_Calvo_Casas_Sanz_Serna:2021} provide a way to improve the efficiency of the Hamiltonian Monte Carlo algorithm by reducing the expensive evaluation of the gradient of the carefully chosen potential.
However, in \Cref{Alg:1}, the generation of the samples in step \eqref{Alg1:step_1} is inexpensive compared to the remaining steps \eqref{Alg1:step_2} and \eqref{Alg1:step_3}.
Therefore, when using splitting schemes for the time integrator of the Hamiltonian Monte Carlo Algorithm \ref{Alg:HMC}, we did not observe any noticeable improvement in \Cref{Alg:1} over the standard St\o rmer-Verlet integrator.

\section{Numerical examples}\label{Sec:5}
This section complements the theoretical results with numerical examples for the three different sampling approaches and the most important operators. 

In all our tests, we assume that the initial state $\psi_0$ is a spherical Gaussian wavepacket $\psi_0 = g_{z_0}$ with $\Gamma=\mathrm{Id}_D$,  centred at a phase-space point $z_0\in\mathbb{R}^{2D}$. 
For all examples, we used the crude Monte Carlo estimators \eqref{EQ:Case_H} for the Husimi approach and \eqref{EQ:Case_sqrt_H} for the sqrt-Husimi approach.
For the optimal approach, when we compute the norm, we used the crude Monte Carlo estimator \eqref{EQ:Case_opt} and for other observables, we used the weighted importance sampling estimator $\eqref{Estimator_WIS_opt}$ in combination with the Hamiltonian Monte Carlo \Cref{Alg:HMC}.

To choose a good weight $W$ for the optimal approach with the weighted importance sampling estimator, we first investigate the initial sampling error for several choices of the weight $W$ in one and six dimensions.
Subsequently, we compare the performance at the initial time for low and high dimensions of the Husimi, sqrt-Husimi, and optimal approaches. 
In the case of a five-dimensional harmonic oscillator, we compare numerical results with analytical solutions, showing that the optimal approach converges faster.
Finally, we examine a six-dimensional Henon-Heiles potential for various initial conditions, going from regions of low to high anharmonicity.
We show that, whenever the Herman--Kluk prefactor blows up on average, the choice of the density is of minor importance.

\subsection{Weight $W$} \label{NumericalExample_Weight}

\begin{figure}[ht!]
    \centering
    \includegraphics{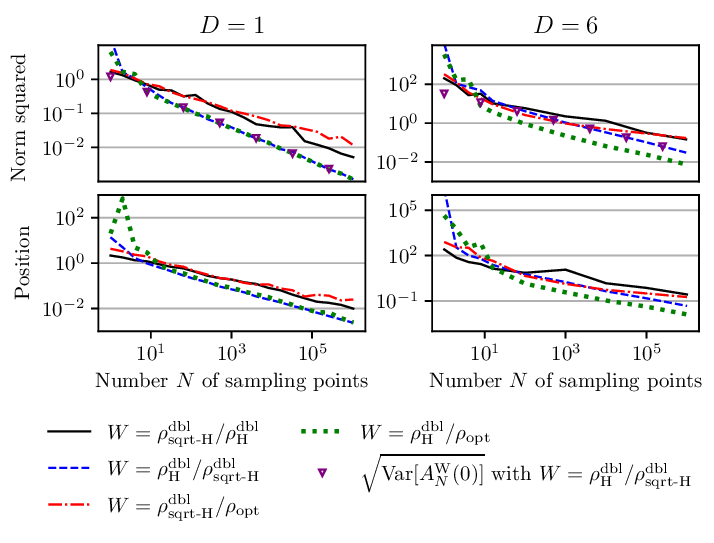}
    \caption{Initial sampling error for $\hat{A}=\mathrm{Id}$ and $\hat{A}=\hat{q}_1$ in one and six dimensions as a function of the number $N$ of Monte Carlo quadrature points. Each panel displays the error for different choices of the weight $W$ as well as the analytical error estimation for $W=\rho_{\rm H}^{\rm dbl}/\rho_{\textup{sqrt-H}}^{\rm dbl}$ and $\hat{A}=\mathrm{Id}$  derived in \Cref{EX:3.6}.}
    \label{fig:1}
\end{figure}

To determine the choice of the weight $W$, we provide the following numerical example.
Consider a Gaussian initial state $\psi_0 = g_{z_0}$ with unit width, $q_0=(1,\dots,1)\in \mathbb{R}^D$ and $p_0=(1,\dots,1)\in \mathbb{R}^D$ and let $\hat{A}=\mathrm{Id}$ or $\hat{q}_1$.
\Cref{fig:1} displays the mean squared error of the weighted importance sampling estimator in one and six dimensions for several choices of the weight $W$.
The value obtained analytically in \Cref{EX:3.6} matches our numerical result.
Moreover, in both dimensions and for both operators, for large enough $N$, the choice $W=\rho^{\rm dbl}_{\rm H}/\rho_{\rm opt}$ for the weight provided the highest accuracy even though, for the position operator, the weight $W$ might have a diverging second moment.
Finally, for small $N$, one can see a higher error due to the bias of the estimator.

This short numerical example suggests using the weighted importance sampling estimator \eqref{EQ:weighted_MC_estimator} with the weight $W=\rho^{\rm dbl}_{\rm H}/\rho_{\rm opt}$.

\subsection{Initial time}
\begin{figure}[ht!]
    \centering
    \includegraphics{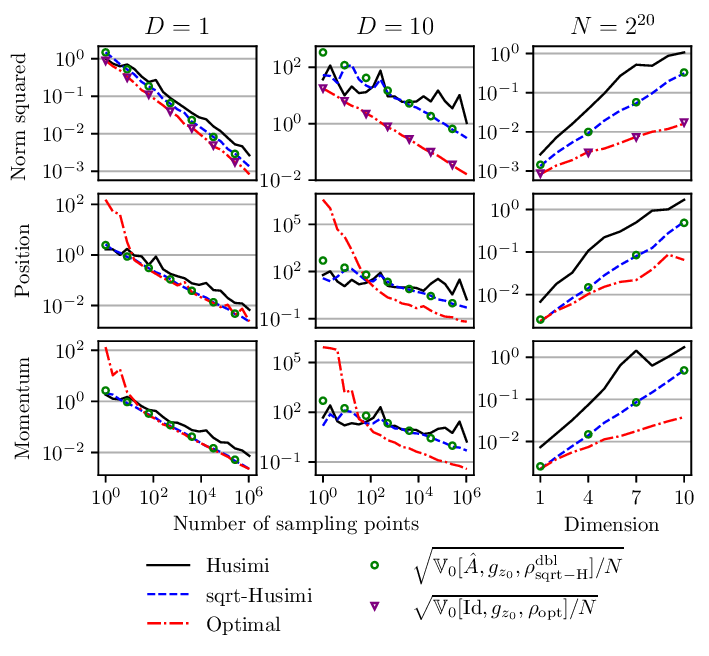}
    \caption{Sampling error of $\hat{A}=\mathrm{Id}$, $\hat{A}=\hat{q}_1$ and $\hat{A}=\hat{p}_1$ in one and ten dimensions as a function of the number $N$ of Monte Carlo points as well as a function of dimension $D$ for a fixed number of samples $N= 2^{20}$. 
    Each panel displays the error for the Husimi (solid line), sqrt-Husimi (dashed line), and optimal approach (dash-dotted line). Theoretical error estimations for the sqrt-Husimi and optimal approaches are displayed with marked lines. }
    \label{fig:Init_time}
\end{figure}
We start by considering the initial sampling error for the norm, position, and momentum expectations, that is, $\hat{A}=\mathrm{Id}, \hat{q}_j$ and $\hat{p}_j$,
with  $\epsilon = 1$, $q_0=(1, \dots , 1)$, and $p_0=(1, \dots ,1)$ for up to 10 dimensions. 
\Cref{fig:Init_time} shows the absolute errors
\begin{align}
    \vert A_N(0) - \langle \hat{A} \rangle_0 \vert
\end{align}
and
\begin{align}
    \vert A_N^{\rm W}(0) - \langle \hat{A} \rangle_0 \vert
\end{align}
between the Monte Carlo approximations $A_N(0)$ and $A_N^W(0)$ and the exact value at time $t=0$ as functions of the number of Monte Carlo quadrature points $N$ in one (first column) and ten (second column) dimensions as well as for a fixed $N=2^{20}$ as a function of the dimension $D$ (third column). Each panel was produced by averaging the error over $100$ independent simulations.

Due to the unbounded variance, the Husimi approach performs the worst, and the sqrt-Husimi and optimal approaches confirm our theoretical error estimations derived in \Cref{Ex:Variance_sqrtH}, \Cref{Ex:Variance_opt} and Appendix \ref{Variances_Initial_Time}.
For a small number of trajectories, the bias of the weighted importance sampling estimator causes a significant error that rapidly decays.

For $D=1$ and large $N$, there is almost no difference between the optimal and the sqrt-Husimi approaches.
However, in high dimensions, starting from approximately $N=10^2$ Monte Carlo points, both the correlation of the chain and the bias of the estimator are negligible, and we obtain an improvement with the optimal approach over both the Husimi and sqrt-Husimi approaches. If we compare different methods ``horizontally'', to reach the same accuracy with the optimal approach, one might need up to a factor of approximately $10^2$ fewer trajectories compared with the sqrt-Husimi approach.

\subsection{Harmonic Potential}
\begin{figure}[ht!]
    \centering
    \includegraphics{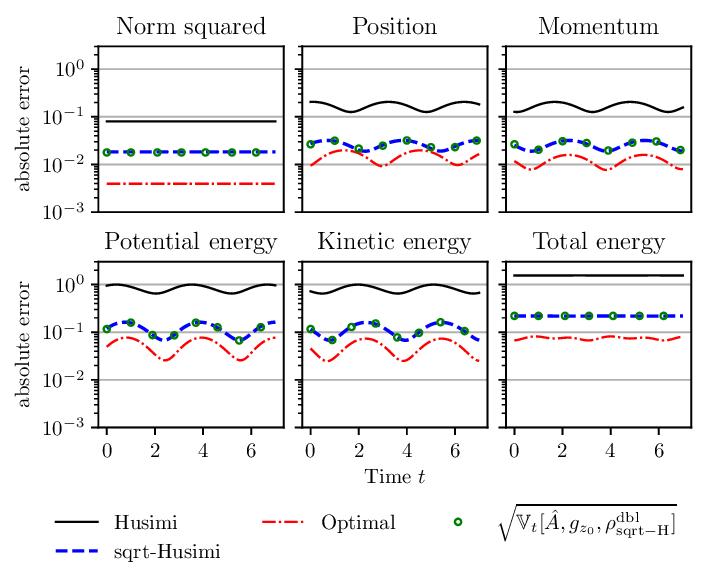}
    \caption{Time dependence of the sampling error of the Herman--Kluk expectation values of norm squared, position, momentum and energies propagated in a harmonic oscillator. 
    Each panel is produced by 100 independent simulations each consisting of $N=2^{20}$ trajectories. }
    \label{fig:Harmonic1}
\end{figure}
To analyse the time dependence of the variance, we consider a harmonic potential $V(x)=\vert x \vert^2/2$ in five spatial dimensions for one full oscillation period. 
We consider the same initial values as in the previous numerical example and compute the norm, position, momentum, potential, kinetic, and total energies and compare them to the corresponding expectation values $\langle\hat{A}\rangle_t$ obtained from the exact solution of the time-dependent Schrödinger equation \cite{book_Heller:2018, Kroeninger_Lasser_Vanicek:2023}. In this numerical example, the exact solutions are given by the constants $\langle \mathrm{Id} \rangle_t = 1$, $\langle \hat{H} \rangle_t = 6D/4 $ and the time-dependent values
\begin{equation}
\begin{alignedat}{2}
\langle \hat{q}_j\rangle_t    &= \cos{t} + \sin{t},
\quad\raisebox{-.5\normalbaselineskip}[0pt][0pt]{for all $j\in  \{1,\dots , D\},$ } \\
\langle \hat{p}_j\rangle_t &= \cos{t} - \sin{t}, \\
\langle V \rangle_t &= \frac{D}{4}\left( 1 + 2(\cos{t} + \sin{t}))^2 \right), \\
\langle T \rangle_t    &= \frac{D}{4}\left( 1 + 2(\cos{t} - \sin{t}))^2 \right).
\end{alignedat}
\end{equation}

\Cref{fig:Harmonic1} shows the behaviour of the absolute error between our Monte Carlo approximations with $N=2^{20}$  quadrature points and the exact solution as a function of time $t$.
The trajectories for the optimal approach were sampled using the Hamiltonian Monte Carlo \Cref{Alg:HMC}. 
The error of the total energy was produced by taking the absolute value of the difference between the exact total energy and the sum of the approximations for the potential and kinetic energy. 
In all the examples, we ran 100 individual simulations and then plotted the square root of the mean of the squares of the errors. 
We can see that, in all cases, the Husimi approach gives the highest error, whereas the optimal approach provides the most accurate results. 
The Husimi approach needs approximately 100 times more trajectories to achieve the same accuracy as the sqrt-Husimi approach and the sqrt-Husimi approach requires approximately 10 times more trajectories to reach the same accuracy as the optimal approach.
Moreover, the result for the square root Husimi approach almost perfectly coincides with the explicitly predicted errors (see Appendix \ref{Variances_Harmonic_Potential}).

\subsection{Henon-Heiles Potential}
Finally, we consider the modified Henon-Heiles potential
\begin{align}\label{Henon-Heiles_Potential}
V(x) = \sum_{j=1}^D \frac{x_j^2}{2} +\sum_{j=1}^{D-1} \left[\sigma \left( x_j x_{j+1}^2 - \frac{x_j^3}{3} \right) +\frac{\sigma^2}{16} \left( x_j^2 + x_{j+1}^2 \right)^2\right], \quad \sigma\geq 0,
\end{align}
as in \cite{Lasser_Sattlegger:2017, Faou_Lubich:2009, Meyer_Cederbaum:1990, Raab_Meyer:2000, Kucar_Cederbaum:1987, Lasser_Roeblitz:2010} in six spatial dimensions. This potential differs from the standard Henon-Heiles model \cite{Henon_Heiles:1964, Kay:1989, Choi_Vanicek:2019a} by an additional quartic term that makes the system bound.
We take $\epsilon=0.01$, $\sigma=1/\sqrt{80}$ and $p_0=(0, \dots,0)$ based on \cite{Lasser_Roeblitz:2010, Lasser_Sattlegger:2017} and let the initial displacement $q_0$ of the Gaussian wavepacket vary to be in regions of the potential with low to high anharmonicity. For the time propagation, we use a second-order St\o rmer-Verlet scheme with a stepsize $\tau = 0.2$.

\begin{figure}
    \centering
    \includegraphics{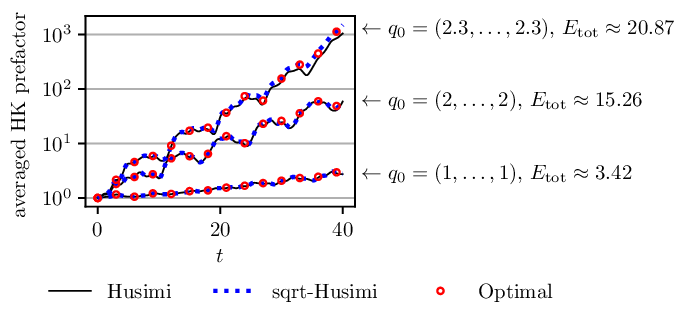}
    \caption{Time dependence of the averaged Herman--Kluk prefactor in a 6-dimensional Henon-Heiles potential for the three different sampling approaches and three different choices of the initial position $q_0$. The averaged Herman--Kluk prefactors were calculated with $N=2^{17}$ quadrature points each. For each initial condition, we display the total energy $E_{\mathrm{tot}}$.}
    \label{Fig:Avg_prefac_energy}
\end{figure}

\Cref{Fig:Avg_prefac_energy} shows the behaviour over time of the averaged Herman--Kluk prefactors
\begin{align}
    \frac{1}{N}\sum_{j=1}^N \vert R_t(y_j) R_t(z_j) \vert, 
\end{align}
where $(y_j,z_j)$ are sampled from $\rho^{\rm dbl}_{\rm H}, \rho^{\rm dbl}_{\textup{sqrt-H}}$ or $\rho_{\rm opt}$. 
We considered three different initial conditions, with  $q_0=k \cdot (1, \dots ,1)$ for $k\in\{1,2,2.3\}$. 
Note that for any initial displacement, the three sampling approaches result in a similar behaviour of the averaged Herman--Kluk prefactors.  
Moreover, for each  $k\in\{1,2,2.3\}$, we display the total energy $E_{\mathrm{tot}}$ of the systems that can be obtained from formulas stated in Appendix~\ref{Appendix:Inner_products}.
Small changes in the total energy of the system lead to large changes in the growth of the Herman--Kluk prefactor, showing that the potential \eqref{Henon-Heiles_Potential} has very anharmonic regions.

In the following, we will analyze the time evolution of the intrinsic error 
\begin{align}\label{EQ:Intrinsic_error}
    \vert A_N(t)-A_{2N}(t)\vert
\end{align} for the different displacements.
\Cref{fig:Henon1} shows the evolution up to the final time $T=40$ for a fixed number $N=2^{19}\approx 5\times 10^5$ of quadrature points. 
For a small $q_0=(1,\dots, 1)$, the optimal approach has a significantly lower error than the square root Husimi approach, and the Husimi approach performs even worse because of the unbounded variance of its estimator \eqref{EQ:Case_H}. 
For $q_0=(2, \dots, 2)$, the Husimi approach continues to have a significantly larger error than the other two approaches. 
The square root Husimi and optimal approaches behave similarly. 
The optimal approach still produces a smaller error, but it is less significant and only for times up to about $T=20$.
For a large $q_0=(2.3, \dots, 2.3)$, the square root Husimi and optimal approaches exhibit a similar performance. 
However, the Husimi approach does not necessarily increase the error, and if it does, then only for small times up to approximately $T=20$.
We can conclude that for small anharmonicity, where the Herman--Kluk prefactor behaves moderately, the choice of the density significantly improves the accuracy. For increasing anharmonicity, and hence for a Herman--Kluk prefactor that blows up, the choice of the density is less important up to the point where even an estimator with unbounded variance has a similar accuracy.

\begin{figure}[ht!]
    \centering
    \includegraphics{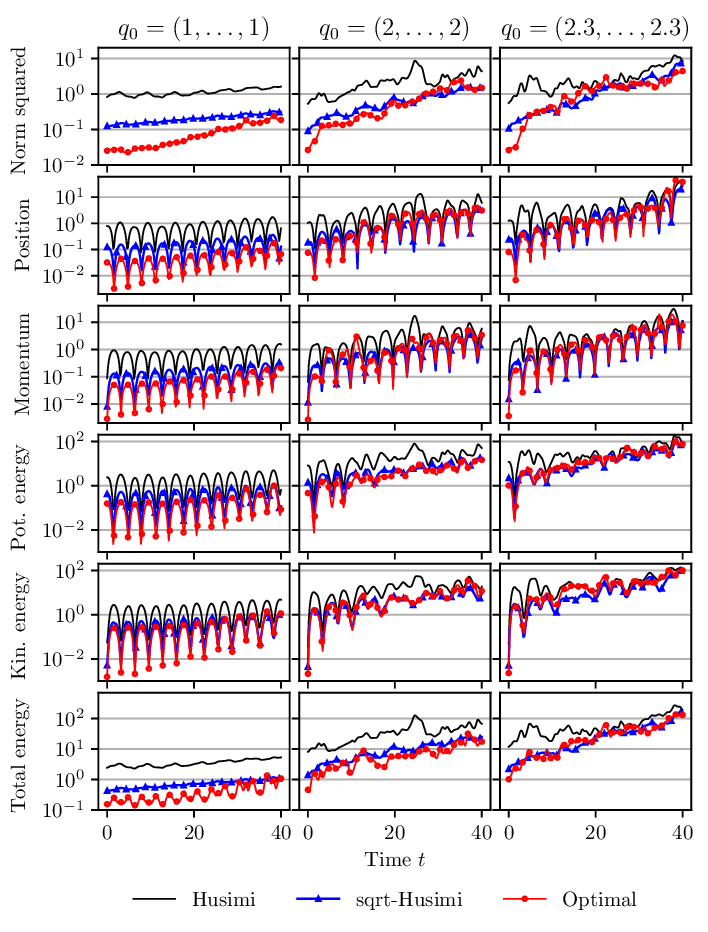}
    \caption{Time evolution of the intrinsic error \eqref{EQ:Intrinsic_error} using $N=2^{19}\approx 5\times 10^5$ quadrature points for the Husimi (solid line), square root Husimi (triangular marker) and optimal (circular marker) approaches and for three different choices of the initial position $q_0$. Each line was produced by averaging over eight independent simulations.}
    \label{fig:Henon1}
\end{figure}

\begin{figure}[ht!]
    \centering
    \includegraphics{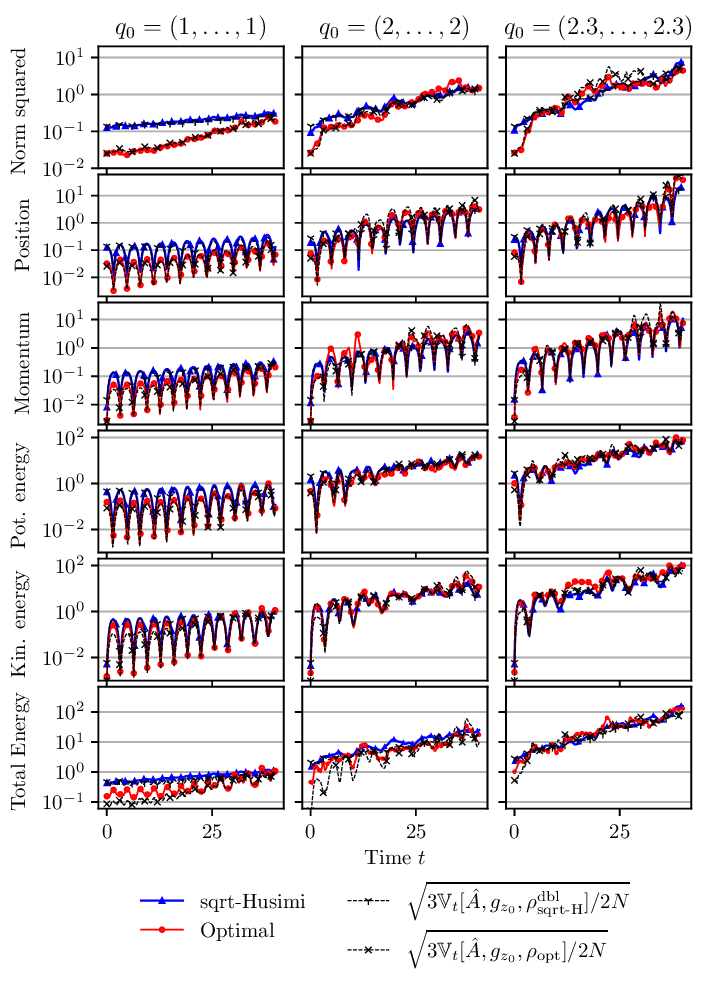}
    \caption{Time evolution of the intrinsic error \eqref{EQ:Intrinsic_error} (solid marked lines) and error estimations \eqref{EQ:Bound_intrinsic_measure} and \eqref{EQ:Bound_intrinsic_measure2} (dashed marked lines) using $N=2^{19}\approx 5\times 10^5$ quadrature points for three different choices of the initial position $q_0$. Each line was produced by averaging over eight independent simulations.}
    \label{fig:Henon2}
\end{figure}

Finally, we consider error estimations of the intrinsic error \eqref{EQ:Intrinsic_error} for the square root Husimi and optimal approaches. 
From \eqref{EQ:Intrinsic_error} it follows that
\begin{align}\label{EQ:Bound_intrinsic_measure}
    \mathbb{E}[\vert A_N(t)-A_{2N}(t)\vert^2] = \frac{3}{2N} \mathbb{V}_t[\hat{A}, \psi_0, \rho^{\rm dbl}_{\textup{\rm sqrt-H}} ],
\end{align}
where the variance can be simultaneously evaluated with another Monte Carlo estimator and without substantial additional effort  (see Appendix~\ref{Appendix:Bound_intrinsic_measure}). 
Even though we know for the weighted importance sampling estimator at most the order of the mean squared error by \Cref{Theorem:Properties_WIS}, we can still estimate $\mathbb{V}_t[\hat{A}, \psi_0, \rho_{\rm opt}]$ and compare the mean squared error with

\begin{align}\label{EQ:Bound_intrinsic_measure2}
     \frac{3}{2N} \mathbb{V}_t[\hat{A}, \psi_0, \rho_{\rm opt}].
\end{align}
More details on \eqref{EQ:Bound_intrinsic_measure} and the estimators of the variances can be found in Appendix~\ref{Appendix:Bound_intrinsic_measure}. 
\Cref{fig:Henon2} shows the same time evolution of the numerical intrinsic error \eqref{EQ:Intrinsic_error} as in \Cref{fig:Henon1}, but only for the square root Husimi and optimal approaches. 
Additionally, in \Cref{fig:Henon2} we also display the probabilistic error estimation \eqref{EQ:Bound_intrinsic_measure} for the sqrt-Husimi approach and the approximate error \eqref{EQ:Bound_intrinsic_measure2} for the optimal approach. 
Interestingly, we observe that in most cases, the weighted importance sampling estimator approximately follows \eqref{EQ:Bound_intrinsic_measure2}.

\section{Conclusion, outlook and further applications}\label{Sec:6}
We presented an analysis of Monte Carlo quadrature techniques for evaluating expectation values of quantum-mechanical observables $\hat{A}$ within the Herman--Kluk approximation.
After introducing a sufficient condition for the convergence of the crude Monte Carlo estimator, we investigated three sampling approaches.
For the initial double phase-space sampling, we used one of three different strategies, based on the Husimi density \eqref{EQ:Case_H}, its square root \eqref{EQ:Case_sqrt_H} or the new optimal method \eqref{EQ:Case_opt}. 
While the generation of quadrature points with the first two approaches is independent of $\hat{A}$, the optimal approach incorporates the quantum-mechanical observable in its sampling.
The Husimi approach has a Monte Carlo integrand with an unbounded second moment, and the square root Husimi approach produces a favourable finite variance.
The optimal approach minimizes the variance of the crude Monte Carlo integrand at the initial time but it still has an exponential dependence on the spatial dimension $D$ (see also \cite[Section 12.2]{Henning_Kersting:2022} for a discussion of the dimensionality dependence of Monte Carlo quadrature).
To take full advantage of the new approach, we extended the crude Monte Carlo estimator to a weighted, self-normalizing version.

The numerical experiments for the harmonic oscillator and the Henon-Heiles potential confirm that the infinite second moment leads to slower convergence of the Husimi approach and that the optimal approach has the fastest convergence. 
For dynamical regions with high anharmonicities, where the Herman--Kluk propagator is known to be inaccurate due to the drastic growth of the Herman--Kluk prefactor, the differences in convergence become smaller up to the point, where the choice of the sampling approach plays a minor role.

The new approach presented in this paper has applications in several fields.
Theoretical chemists make ubiquitous use of the Herman--Kluk approximation. 
They apply further approximations such as time-averaging \cite{Elran_Kay:1999a, Kaledin_Miller:2003, Buchholz_Ceotto:2016, Buchholz_Ceotto:2018}, Filinov filtering \cite{Walton_Manolopoulos:1996, Filinov:1986, Makri_Miller:1987}, hybrid dynamics \cite{Grossmann:2006, Goletz_Grossmann:2009} and many others to improve Monte Carlo estimators by reducing the required number of quadrature points. 
Our idea can be combined with these techniques to make the Herman--Kluk approximation even more computationally attractive.
Moreover, expectation values are a special case of time-correlation functions \cite{Church_Ananth:2017, Pollak_Liu:2022, Miller:2001, Sun_Miller:2002}. 
The idea of including observables in the sampling can be extended to time-correlation functions. It has been explored for classical correlation functions in \cite{Zimmermann_Vanicek:2013} and will be explored for evaluating time-correlation functions within the Herman--Kluk approximation in the future.
Finally, the analysis at the initial time $t=0$ presented here might be of general interest for all Gaussian-based methods.

\section*{Acknowledgments}
C. Lasser acknowledges financial support by the Deutsche Forschungsgemeinschaft (DFG, German Research Foundation)  TRR 352 – Project-ID 470903074.
F. Kröninger and J. Van\'{i}\v{c}ek acknowledge financial support from the European Research Council (ERC) under the European Union's Horizon 2020 Research and Innovation Programme (Grant Agreement No. $683069$-MOLEQULE).

\bibliographystyle{abbrvnat}
\bibliography{final_bib}


\begin{thebibliography}{75}
\ifx \bisbn   \undefined \def \bisbn  #1{ISBN #1}\fi
\ifx \binits  \undefined \def \binits#1{#1}\fi
\ifx \bauthor  \undefined \def \bauthor#1{#1}\fi
\ifx \batitle  \undefined \def \batitle#1{#1}\fi
\ifx \bjtitle  \undefined \def \bjtitle#1{#1}\fi
\ifx \bvolume  \undefined \def \bvolume#1{\textbf{#1}}\fi
\ifx \byear  \undefined \def \byear#1{#1}\fi
\ifx \bissue  \undefined \def \bissue#1{#1}\fi
\ifx \bfpage  \undefined \def \bfpage#1{#1}\fi
\ifx \blpage  \undefined \def \blpage #1{#1}\fi
\ifx \burl  \undefined \def \burl#1{\textsf{#1}}\fi
\ifx \doiurl  \undefined \def \doiurl#1{\url{https://doi.org/#1}}\fi
\ifx \betal  \undefined \def \betal{\textit{et al.}}\fi
\ifx \binstitute  \undefined \def \binstitute#1{#1}\fi
\ifx \binstitutionaled  \undefined \def \binstitutionaled#1{#1}\fi
\ifx \bctitle  \undefined \def \bctitle#1{#1}\fi
\ifx \beditor  \undefined \def \beditor#1{#1}\fi
\ifx \bpublisher  \undefined \def \bpublisher#1{#1}\fi
\ifx \bbtitle  \undefined \def \bbtitle#1{#1}\fi
\ifx \bedition  \undefined \def \bedition#1{#1}\fi
\ifx \bseriesno  \undefined \def \bseriesno#1{#1}\fi
\ifx \blocation  \undefined \def \blocation#1{#1}\fi
\ifx \bsertitle  \undefined \def \bsertitle#1{#1}\fi
\ifx \bsnm \undefined \def \bsnm#1{#1}\fi
\ifx \bsuffix \undefined \def \bsuffix#1{#1}\fi
\ifx \bparticle \undefined \def \bparticle#1{#1}\fi
\ifx \barticle \undefined \def \barticle#1{#1}\fi
\bibcommenthead
\ifx \bconfdate \undefined \def \bconfdate #1{#1}\fi
\ifx \botherref \undefined \def \botherref #1{#1}\fi
\ifx \url \undefined \def \url#1{\textsf{#1}}\fi
\ifx \bchapter \undefined \def \bchapter#1{#1}\fi
\ifx \bbook \undefined \def \bbook#1{#1}\fi
\ifx \bcomment \undefined \def \bcomment#1{#1}\fi
\ifx \oauthor \undefined \def \oauthor#1{#1}\fi
\ifx \citeauthoryear \undefined \def \citeauthoryear#1{#1}\fi
\ifx \endbibitem  \undefined \def \endbibitem {}\fi
\ifx \bconflocation  \undefined \def \bconflocation#1{#1}\fi
\ifx \arxivurl  \undefined \def \arxivurl#1{\textsf{#1}}\fi
\csname PreBibitemsHook\endcsname

\bibitem[\protect\citeauthoryear{Born and Oppenheimer}{1927}]{Born_Oppenheimer:1927}
\begin{barticle}
\bauthor{\bsnm{Born}, \binits{M.}},
\bauthor{\bsnm{Oppenheimer}, \binits{R.}}:
\batitle{{Zur {Q}uantentheorie der {M}olekeln}}.
\bjtitle{Ann.~d.~Phys.}
\bvolume{389}(\bissue{20}),
\bfpage{457}--\blpage{484}
(\byear{1927})
\end{barticle}
\endbibitem

\bibitem[\protect\citeauthoryear{Robert and Combescure}{2021}]{Robert_Combescure:2021}
\begin{bbook}
\bauthor{\bsnm{Robert}, \binits{D.}},
\bauthor{\bsnm{Combescure}, \binits{M.}}:
\bbtitle{{Coherent States and Applications in Mathematical Physics}}.
\bpublisher{Springer},
\blocation{Switzerland}
(\byear{2021})
\end{bbook}
\endbibitem

\bibitem[\protect\citeauthoryear{Miller}{2001}]{Miller:2001}
\begin{barticle}
\bauthor{\bsnm{Miller}, \binits{W.H.}}:
\batitle{{The Semiclassical Initial Value Representation: A Potentially Practical Way for Adding Quantum Effects to Classical Molecular Dynamics Simulations}}.
\bjtitle{J.~Phys.\ Chem.~A}
\bvolume{105}(\bissue{13}),
\bfpage{2942}
(\byear{2001})
\end{barticle}
\endbibitem

\bibitem[\protect\citeauthoryear{Tannor}{2007}]{book_Tannor:2007}
\begin{bbook}
\bauthor{\bsnm{Tannor}, \binits{D.J.}}:
\bbtitle{{Introduction to Quantum Mechanics: A Time-Dependent Perspective}}.
\bpublisher{University Science Books},
\blocation{Sausalito}
(\byear{2007})
\end{bbook}
\endbibitem

\bibitem[\protect\citeauthoryear{Lubich}{2008}]{book_Lubich:2008}
\begin{bbook}
\bauthor{\bsnm{Lubich}, \binits{C.}}:
\bbtitle{{From Quantum to Classical Molecular Dynamics: Reduced Models and Numerical Analysis}},
\bedition{12} edn.
\bpublisher{European Mathematical Society},
\blocation{Z\"{u}rich}
(\byear{2008})
\end{bbook}
\endbibitem

\bibitem[\protect\citeauthoryear{Heller}{2018}]{book_Heller:2018}
\begin{bbook}
\bauthor{\bsnm{Heller}, \binits{E.J.}}:
\bbtitle{{The Semiclassical Way to Dynamics and Spectroscopy}}.
\bpublisher{Princeton University Press},
\blocation{Princeton, NJ}
(\byear{2018})
\end{bbook}
\endbibitem

\bibitem[\protect\citeauthoryear{Lasser and Lubich}{2020}]{Lasser_Lubich:2020}
\begin{barticle}
\bauthor{\bsnm{Lasser}, \binits{C.}},
\bauthor{\bsnm{Lubich}, \binits{C.}}:
\batitle{{Computing quantum dynamics in the semiclassical regime}}.
\bjtitle{Acta Numerica}
\bvolume{29},
\bfpage{229}--\blpage{401}
(\byear{2020})
\end{barticle}
\endbibitem

\bibitem[\protect\citeauthoryear{Heller}{1975}]{Heller:1975}
\begin{barticle}
\bauthor{\bsnm{Heller}, \binits{E.J.}}:
\batitle{{Time-dependent approach to semiclassical dynamics}}.
\bjtitle{J.~Chem.\ Phys.}
\bvolume{62}(\bissue{4}),
\bfpage{1544}--\blpage{1555}
(\byear{1975})
\end{barticle}
\endbibitem

\bibitem[\protect\citeauthoryear{Herman and Kluk}{1984}]{Herman_Kluk:1984}
\begin{barticle}
\bauthor{\bsnm{Herman}, \binits{M.F.}},
\bauthor{\bsnm{Kluk}, \binits{E.}}:
\batitle{{A semiclasical justification for the use of non-spreading wavepackets in dynamics calculations}}.
\bjtitle{Chem.\ Phys.}
\bvolume{91}(\bissue{1}),
\bfpage{27}--\blpage{34}
(\byear{1984})
\end{barticle}
\endbibitem

\bibitem[\protect\citeauthoryear{Heller}{1981}]{Heller:1981}
\begin{barticle}
\bauthor{\bsnm{Heller}, \binits{E.J.}}:
\batitle{{Frozen Gaussians: A very simple semiclassical approximation}}.
\bjtitle{J.~Chem.\ Phys.}
\bvolume{75}(\bissue{6}),
\bfpage{2923}--\blpage{2931}
(\byear{1981})
\end{barticle}
\endbibitem

\bibitem[\protect\citeauthoryear{{Ozorio de Almeida} et~al.}{2021}]{Ozorio_Ingold:2021}
\begin{barticle}
\bauthor{\bsnm{{Ozorio de Almeida}}, \binits{A.M.}},
\bauthor{\bsnm{Ingold}, \binits{G.-L.}},
\bauthor{\bsnm{Brodier}, \binits{O.}}:
\batitle{{The quantum canonical ensemble in phase space}}.
\bjtitle{Physica D: Nonlinear Phenomena}
\bvolume{424},
\bfpage{132951}
(\byear{2021})
\end{barticle}
\endbibitem

\bibitem[\protect\citeauthoryear{Littlejohn}{1990}]{Littlejohn:1990}
\begin{barticle}
\bauthor{\bsnm{Littlejohn}, \binits{R.G.}}:
\batitle{{Semiclassical structure of trace formulas}}.
\bjtitle{J. Math. Phys.}
\bvolume{31},
\bfpage{2952}--\blpage{2977}
(\byear{1990})
\end{barticle}
\endbibitem

\bibitem[\protect\citeauthoryear{Antipov et~al.}{2015}]{Antipov_Ananth:2015}
\begin{barticle}
\bauthor{\bsnm{Antipov}, \binits{S.V.}},
\bauthor{\bsnm{Ye}, \binits{Z.}},
\bauthor{\bsnm{Ananth}, \binits{N.}}:
\batitle{{Dynamically consistent method for mixed quantum-classical simulations: a semiclassical approach}}.
\bjtitle{J.~Chem.\ Phys.}
\bvolume{142},
\bfpage{184102}
(\byear{2015})
\end{barticle}
\endbibitem

\bibitem[\protect\citeauthoryear{Makri and Miller}{2002}]{Makri_Miller:2002}
\begin{barticle}
\bauthor{\bsnm{Makri}, \binits{N.}},
\bauthor{\bsnm{Miller}, \binits{W.H.}}:
\batitle{{Coherent state semiclassical initial value representation for the Boltzmann operator in thermal correlation functions}}.
\bjtitle{J.~Chem.\ Phys.}
\bvolume{116}(\bissue{21}),
\bfpage{9207}--\blpage{9212}
(\byear{2002})
\end{barticle}
\endbibitem

\bibitem[\protect\citeauthoryear{{Ozorio de Almeida} and Brodier}{2006}]{Ozorio_Brodier:2006}
\begin{barticle}
\bauthor{\bsnm{{Ozorio de Almeida}}, \binits{A.M.}},
\bauthor{\bsnm{Brodier}, \binits{O.}}:
\batitle{{Phase space propagators for quantum operators}}.
\bjtitle{Annals of Physics}
\bvolume{321}(\bissue{8}),
\bfpage{1790}--\blpage{1813}
(\byear{2006})
\end{barticle}
\endbibitem

\bibitem[\protect\citeauthoryear{Saraceno and {Ozorio de Almeida}}{2016}]{Saraceno:2016}
\begin{barticle}
\bauthor{\bsnm{Saraceno}, \binits{M.}},
\bauthor{\bsnm{{Ozorio de Almeida}}, \binits{A.M.}}:
\batitle{{Representation of superoperators in double phase space}}.
\bjtitle{J.~Phys.~A}
\bvolume{49}(\bissue{14}),
\bfpage{145302}
(\byear{2016})
\end{barticle}
\endbibitem

\bibitem[\protect\citeauthoryear{Church et~al.}{2017}]{Church_Ananth:2017}
\begin{barticle}
\bauthor{\bsnm{Church}, \binits{M.S.}},
\bauthor{\bsnm{Antipov}, \binits{S.V.}},
\bauthor{\bsnm{Ananth}, \binits{N.}}:
\batitle{{Validating and implementing modified Filinov Phase Filtration in Semiclassical Dynamics}}.
\bjtitle{J.~Chem.\ Phys.}
\bvolume{146},
\bfpage{234104}
(\byear{2017})
\end{barticle}
\endbibitem

\bibitem[\protect\citeauthoryear{Walton and Manolopoulos}{1996}]{Walton_Manolopoulos:1996}
\begin{barticle}
\bauthor{\bsnm{Walton}, \binits{A.R.}},
\bauthor{\bsnm{Manolopoulos}, \binits{D.E.}}:
\batitle{{A new semiclassical initial value method for Franck-Condon spectra}}.
\bjtitle{Mol.\ Phys.}
\bvolume{87}(\bissue{4}),
\bfpage{961}--\blpage{978}
(\byear{1996})
\end{barticle}
\endbibitem

\bibitem[\protect\citeauthoryear{Elran and Kay}{1999}]{Elran_Kay:1999a}
\begin{barticle}
\bauthor{\bsnm{Elran}, \binits{Y.}},
\bauthor{\bsnm{Kay}, \binits{K.G.}}:
\batitle{{Improving the efficiency of the Herman--Kluk propagator by time integration}}.
\bjtitle{J.~Chem.\ Phys.}
\bvolume{110}(\bissue{8}),
\bfpage{3653}--\blpage{3659}
(\byear{1999})
\end{barticle}
\endbibitem

\bibitem[\protect\citeauthoryear{Ceotto et~al.}{2009}]{Ceotto_Atahan:2009a}
\begin{barticle}
\bauthor{\bsnm{Ceotto}, \binits{M.}},
\bauthor{\bsnm{Atahan}, \binits{S.}},
\bauthor{\bsnm{Tantardini}, \binits{G.F.}},
\bauthor{\bsnm{Aspuru-Guzik}, \binits{A.}}:
\batitle{{Multiple coherent states for first-principles semiclassical initial value representation molecular dynamics}}.
\bjtitle{J.~Chem.\ Phys.}
\bvolume{130}(\bissue{23}),
\bfpage{234113}
(\byear{2009})
\end{barticle}
\endbibitem

\bibitem[\protect\citeauthoryear{Makri}{2011}]{Makri:2011}
\begin{barticle}
\bauthor{\bsnm{Makri}, \binits{N.}}:
\batitle{{Forward-backward semiclassical and quantum trajectory methods for time correlation functions}}.
\bjtitle{Phys.\ Chem.\ Chem.\ Phys.}
\bvolume{13},
\bfpage{14442}--\blpage{14452}
(\byear{2011})
\end{barticle}
\endbibitem

\bibitem[\protect\citeauthoryear{Zimmermann and Vaníček}{2013}]{Zimmermann_Vanicek:2013}
\begin{barticle}
\bauthor{\bsnm{Zimmermann}, \binits{T.}},
\bauthor{\bsnm{Vaníček}, \binits{J.}}:
\batitle{{Role of sampling in evaluating classical time autocorrelation functions}}.
\bjtitle{J.~Chem.\ Phys.}
\bvolume{139},
\bfpage{104105}
(\byear{2013})
\end{barticle}
\endbibitem

\bibitem[\protect\citeauthoryear{Swart and Rousse}{2008}]{Swart_Rousse:2008}
\begin{barticle}
\bauthor{\bsnm{Swart}, \binits{T.}},
\bauthor{\bsnm{Rousse}, \binits{V.}}:
\batitle{{A Mathematical Justification for the Herman-Kluk Propagator}}.
\bjtitle{Commun. Math. Phys.}
\bvolume{286},
\bfpage{725}--\blpage{750}
(\byear{2008})
\end{barticle}
\endbibitem

\bibitem[\protect\citeauthoryear{Robert}{2010}]{Robert2010}
\begin{barticle}
\bauthor{\bsnm{Robert}, \binits{D.}}:
\batitle{{On the Herman--Kluk Semiclassical Approximation}}.
\bjtitle{Reviews in Mathematical Physics}
\bvolume{22}(\bissue{10}),
\bfpage{1123}--\blpage{1145}
(\byear{2010})
\end{barticle}
\endbibitem

\bibitem[\protect\citeauthoryear{Lu and Yang}{2010}]{Lu_Yang:2010}
\begin{barticle}
\bauthor{\bsnm{Lu}, \binits{J.}},
\bauthor{\bsnm{Yang}, \binits{X.}}:
\batitle{{Frozen {G}aussian approximation for high frequency wave propagation}}.
\bjtitle{Communications in Mathematical Sciences}
\bvolume{9},
\bfpage{663}--\blpage{683}
(\byear{2010})
\end{barticle}
\endbibitem

\bibitem[\protect\citeauthoryear{Lu and Yang}{2012a}]{Lu_Yang:2012}
\begin{barticle}
\bauthor{\bsnm{Lu}, \binits{J.}},
\bauthor{\bsnm{Yang}, \binits{X.}}:
\batitle{{Convergence of {F}rozen {G}aussian {A}pproximation for {H}igh-{F}requency {W}ave {P}ropagation}}.
\bjtitle{Communications on Pure and Applied Mathematics}
\bvolume{65}(\bissue{6}),
\bfpage{759}--\blpage{789}
(\byear{2012})
\end{barticle}
\endbibitem

\bibitem[\protect\citeauthoryear{Lu and Yang}{2012b}]{Lu_Yang:2012a}
\begin{barticle}
\bauthor{\bsnm{Lu}, \binits{J.}},
\bauthor{\bsnm{Yang}, \binits{X.}}:
\batitle{{Frozen Gaussian Approximation for General Linear Strictly Hyperbolic Systems: Formulation and Eulerian Methods}}.
\bjtitle{SIAM Multiscale Modeling \& Simulation}
\bvolume{10}(\bissue{2}),
\bfpage{451}--\blpage{472}
(\byear{2012})
\end{barticle}
\endbibitem

\bibitem[\protect\citeauthoryear{Delgadillo et~al.}{2016}]{Delgadillo_Yang:2016}
\begin{barticle}
\bauthor{\bsnm{Delgadillo}, \binits{R.}},
\bauthor{\bsnm{Lu}, \binits{J.}},
\bauthor{\bsnm{Yang}, \binits{X.}}:
\batitle{{Gauge-Invariant Frozen Gaussian Approximation Method for the Schrödinger Equation with Periodic Potentials}}.
\bjtitle{SIAM Journal on Scientific Computing}
\bvolume{38}(\bissue{4}),
\bfpage{2440}--\blpage{2463}
(\byear{2016})
\end{barticle}
\endbibitem

\bibitem[\protect\citeauthoryear{Delgadillo et~al.}{2018}]{Delgadillo_Yang:2018}
\begin{barticle}
\bauthor{\bsnm{Delgadillo}, \binits{R.}},
\bauthor{\bsnm{Lu}, \binits{J.}},
\bauthor{\bsnm{Yang}, \binits{X.}}:
\batitle{{Frozen Gaussian approximation for high frequency wave propagation in periodic media}}.
\bjtitle{Asymptotic Analysis}
\bvolume{110},
\bfpage{113}--\blpage{135}
(\byear{2018})
\end{barticle}
\endbibitem

\bibitem[\protect\citeauthoryear{Lasser and Sattlegger}{2017}]{Lasser_Sattlegger:2017}
\begin{barticle}
\bauthor{\bsnm{Lasser}, \binits{C.}},
\bauthor{\bsnm{Sattlegger}, \binits{D.}}:
\batitle{{Discretising the Herman-Kluk propagator}}.
\bjtitle{Numerische Mathematik}
\bvolume{137},
\bfpage{119}--\blpage{157}
(\byear{2017})
\end{barticle}
\endbibitem

\bibitem[\protect\citeauthoryear{Kröninger et~al.}{2023}]{Kroeninger_Lasser_Vanicek:2023}
\begin{botherref}
\oauthor{\bsnm{Kröninger}, \binits{F.}},
\oauthor{\bsnm{Lasser}, \binits{C.}},
\oauthor{\bsnm{Vaníček}, \binits{J.}}:
{Sampling strategies for the Herman–Kluk propagator of the wavefunction}.
Frontiers in Physics
\textbf{11}
(2023)
\end{botherref}
\endbibitem

\bibitem[\protect\citeauthoryear{Xie and Zhou}{2023}]{Xie_Zhou:2021}
\begin{barticle}
\bauthor{\bsnm{Xie}, \binits{Y.}},
\bauthor{\bsnm{Zhou}, \binits{Z.}}:
\batitle{{Frozen Gaussian Sampling: A Mesh-free Monte Carlo Method For Approximating Semiclassical Schrödinger Equations}}.
\bjtitle{Commun. Math. Sci.}
\bvolume{22}(\bissue{4}),
\bfpage{1133}--\blpage{1166}
(\byear{2023})
\end{barticle}
\endbibitem

\bibitem[\protect\citeauthoryear{Huang et~al.}{2022}]{Huang_Zhou:2022}
\begin{barticle}
\bauthor{\bsnm{Huang}, \binits{Z.}},
\bauthor{\bsnm{Xu}, \binits{L.}},
\bauthor{\bsnm{Zhou}, \binits{Z.}}:
\batitle{{Efficient Frozen Gaussian Sampling algorithms for nonadiabatic quantum dynamics at metal surfaces}}.
\bjtitle{J.~Comp.\ Phys.}
\bvolume{474},
\bfpage{111771}
(\byear{2022})
\end{barticle}
\endbibitem

\bibitem[\protect\citeauthoryear{Hennig et~al.}{2022}]{Henning_Kersting:2022}
\begin{bbook}
\bauthor{\bsnm{Hennig}, \binits{P.}},
\bauthor{\bsnm{Osborne}, \binits{M.A.}},
\bauthor{\bsnm{Kersting}, \binits{H.P.}}:
\bbtitle{{Probabilistic Numerics: Computation as Machine Learning}}.
\bpublisher{Cambridge University Press},
\blocation{Cambridge}
(\byear{2022})
\end{bbook}
\endbibitem

\bibitem[\protect\citeauthoryear{Martinez}{2002}]{Martinez:2002}
\begin{bbook}
\bauthor{\bsnm{Martinez}, \binits{A.}}:
\bbtitle{{An Introduction to Semiclassical and Microlocal Analysis}}.
\bpublisher{Springer},
\blocation{New York}
(\byear{2002})
\end{bbook}
\endbibitem

\bibitem[\protect\citeauthoryear{Kay}{2006}]{Kay:2006}
\begin{barticle}
\bauthor{\bsnm{Kay}, \binits{K.G.}}:
\batitle{{The Herman–Kluk approximation: Derivation and semiclassical corrections}}.
\bjtitle{Chem. Phys.}
\bvolume{322},
\bfpage{3}--\blpage{12}
(\byear{2006})
\end{barticle}
\endbibitem

\bibitem[\protect\citeauthoryear{Hall}{2013}]{Hall:2013}
\begin{bbook}
\bauthor{\bsnm{Hall}, \binits{B.C.}}:
\bbtitle{{Quantum Theory for Mathematicians}}.
\bpublisher{Springer},
\blocation{NY}
(\byear{2013})
\end{bbook}
\endbibitem

\bibitem[\protect\citeauthoryear{Pollak et~al.}{2022}]{Pollak_Liu:2022}
\begin{barticle}
\bauthor{\bsnm{Pollak}, \binits{E.}},
\bauthor{\bsnm{Upadhyayula}, \binits{S.}},
\bauthor{\bsnm{Liu}, \binits{J.}}:
\batitle{{Coherent state representation of thermal correlation functions with applications to rate theory}}.
\bjtitle{J.~Chem.\ Phys.}
\bvolume{156}(\bissue{24}),
\bfpage{244101}
(\byear{2022})
\end{barticle}
\endbibitem

\bibitem[\protect\citeauthoryear{Sun and Miller}{2002}]{Sun_Miller:2002}
\begin{barticle}
\bauthor{\bsnm{Sun}, \binits{S.X.}},
\bauthor{\bsnm{Miller}, \binits{W.H.}}:
\batitle{{Statistical sampling of semiclassical distributions: Calculating quantum mechanical effects using Metropolis Monte Carlo}}.
\bjtitle{J.~Chem.\ Phys.}
\bvolume{117}(\bissue{12}),
\bfpage{5522}--\blpage{5528}
(\byear{2002})
\end{barticle}
\endbibitem

\bibitem[\protect\citeauthoryear{Hagedorn}{1980}]{Hagedorn:1980}
\begin{barticle}
\bauthor{\bsnm{Hagedorn}, \binits{G.A.}}:
\batitle{{{Semiclassical quantum mechanics. I. The $\hbar \rightarrow 0$ limit for coherent states}}}.
\bjtitle{Commun.\ Math.\ Phys.}
\bvolume{71}(\bissue{1}),
\bfpage{77}--\blpage{93}
(\byear{1980})
\end{barticle}
\endbibitem

\bibitem[\protect\citeauthoryear{Ohsawa}{2021}]{Ohsawa:2021}
\begin{botherref}
\oauthor{\bsnm{Ohsawa}, \binits{T.}}:
{Approximation of semiclassical expectation values by symplectic {G}aussian wave packet dynamics}.
Lett Math Phys
\textbf{111}(121)
(2021)
\end{botherref}
\endbibitem

\bibitem[\protect\citeauthoryear{Burkhard et~al.}{2024}]{Burkhard_Lasser:2023}
\begin{barticle}
\bauthor{\bsnm{Burkhard}, \binits{S.}},
\bauthor{\bsnm{Dörich}, \binits{B.}},
\bauthor{\bsnm{Hochbruck}, \binits{M.}},
\bauthor{\bsnm{Lasser}, \binits{C.}}:
\batitle{{Variational Gaussian approximation for the magnetic Schrödinger equation}}.
\bjtitle{Journal of Physics A: Mathematical and Theoretical}
\bvolume{57}(\bissue{29}),
\bfpage{295202}
(\byear{2024})
\end{barticle}
\endbibitem

\bibitem[\protect\citeauthoryear{Brewer et~al.}{1997}]{Brewer_Hulme:1997}
\begin{barticle}
\bauthor{\bsnm{Brewer}, \binits{M.L.}},
\bauthor{\bsnm{Hulme}, \binits{J.S.}},
\bauthor{\bsnm{Manolopoulos}, \binits{D.E.}}:
\batitle{{Semiclassical dynamics in up to 15 coupled vibrational degrees of freedom}}.
\bjtitle{J.~Chem.\ Phys.}
\bvolume{106}(\bissue{12}),
\bfpage{4832}--\blpage{4839}
(\byear{1997})
\end{barticle}
\endbibitem

\bibitem[\protect\citeauthoryear{Hairer et~al.}{2003}]{hairer_lubich_wanner:2003}
\begin{barticle}
\bauthor{\bsnm{Hairer}, \binits{E.}},
\bauthor{\bsnm{Lubich}, \binits{C.}},
\bauthor{\bsnm{Wanner}, \binits{G.}}:
\batitle{{Geometric numerical integration illustrated by the {S}t{\o}rmer-{V}erlet method}}.
\bjtitle{Acta Numerica}
\bvolume{12},
\bfpage{399}--\blpage{450}
(\byear{2003})
\end{barticle}
\endbibitem

\bibitem[\protect\citeauthoryear{Hairer et~al.}{2006}]{book_Hairer_Wanner:2006}
\begin{bbook}
\bauthor{\bsnm{Hairer}, \binits{E.}},
\bauthor{\bsnm{Lubich}, \binits{C.}},
\bauthor{\bsnm{Wanner}, \binits{G.}}:
\bbtitle{{Geometric Numerical Integration: Structure-Preserving Algorithms for Ordinary Differential Equations}}.
\bpublisher{Springer},
\blocation{Heidelberg}
(\byear{2006})
\end{bbook}
\endbibitem

\bibitem[\protect\citeauthoryear{Caflisch}{1998}]{Caflisch:1998}
\begin{barticle}
\bauthor{\bsnm{Caflisch}, \binits{R.E.}}:
\batitle{{Monte Carlo and quasi-Monte Carlo methods}}.
\bjtitle{Acta Numerica}
\bvolume{7},
\bfpage{1}--\blpage{49}
(\byear{1998})
\end{barticle}
\endbibitem

\bibitem[\protect\citeauthoryear{Liu}{2004}]{Liu:2004}
\begin{bbook}
\bauthor{\bsnm{Liu}, \binits{J.S.}}:
\bbtitle{{Monte Carlo Strategies in Scientific Computing}}.
\bpublisher{Springer},
\blocation{New York}
(\byear{2004})
\end{bbook}
\endbibitem

\bibitem[\protect\citeauthoryear{MacKay}{2005}]{MacKay:2005}
\begin{bbook}
\bauthor{\bsnm{MacKay}, \binits{D.}}:
\bbtitle{{Information Theory, Inference And Learning Algorithms}}.
\bpublisher{Cambridge University Press},
\blocation{Cambridge}
(\byear{2005})
\end{bbook}
\endbibitem

\bibitem[\protect\citeauthoryear{Durrett}{2019}]{durrett:2019}
\begin{bbook}
\bauthor{\bsnm{Durrett}, \binits{R.}}:
\bbtitle{{Probability: Theory and Examples}}.
\bpublisher{Cambridge University Press},
\blocation{Cambridge}
(\byear{2019})
\end{bbook}
\endbibitem

\bibitem[\protect\citeauthoryear{Conte et~al.}{2019}]{Conte_Ceotto:2019}
\begin{barticle}
\bauthor{\bsnm{Conte}, \binits{R.}},
\bauthor{\bsnm{Gabas}, \binits{F.}},
\bauthor{\bsnm{Botti}, \binits{G.}},
\bauthor{\bsnm{Zhuang}, \binits{Y.}},
\bauthor{\bsnm{Ceotto}, \binits{M.}}:
\batitle{{{Semiclassical vibrational spectroscopy with Hessian databases}}}.
\bjtitle{J.~Chem.\ Phys.}
\bvolume{150}(\bissue{24}),
\bfpage{244118}
(\byear{2019})
\end{barticle}
\endbibitem

\bibitem[\protect\citeauthoryear{Kluk et~al.}{1986}]{Kluk_Davis:1986}
\begin{barticle}
\bauthor{\bsnm{Kluk}, \binits{E.}},
\bauthor{\bsnm{Herman}, \binits{M.F.}},
\bauthor{\bsnm{Davis}, \binits{H.L.}}:
\batitle{{Comparison of the propagation of semiclassical frozen Gaussian wave functions with quantum propagation for a highly excited anharmonic oscillator}}.
\bjtitle{J.~Chem.\ Phys.}
\bvolume{84}(\bissue{1}),
\bfpage{326}--\blpage{334}
(\byear{1986})
\end{barticle}
\endbibitem

\bibitem[\protect\citeauthoryear{Troutman}{1995}]{Troutman:1995}
\begin{bbook}
\bauthor{\bsnm{Troutman}, \binits{J.L.}}:
\bbtitle{{Variational Calculus and Optimal Control}}.
\bpublisher{Springer},
\blocation{New York}
(\byear{1995})
\end{bbook}
\endbibitem

\bibitem[\protect\citeauthoryear{Kroese et~al.}{2011}]{Kroese:2011}
\begin{bchapter}
\bauthor{\bsnm{Kroese}, \binits{D.P.}},
\bauthor{\bsnm{Taimre}, \binits{T.}},
\bauthor{\bsnm{Botev}, \binits{Z.I.}}:
\bctitle{9}.
\bbtitle{{Handbook of Monte Carlo Methods}},
pp. \bfpage{347}--\blpage{380}.
\bpublisher{John Wiley \& Sons, Ltd},
\blocation{New York}
(\byear{2011})
\end{bchapter}
\endbibitem

\bibitem[\protect\citeauthoryear{Shiryaev}{2021}]{Shiryaev:2021}
\begin{bbook}
\bauthor{\bsnm{Shiryaev}, \binits{A.N.}}:
\bbtitle{{Probability-2}}.
\bpublisher{Springer},
\blocation{NY}
(\byear{2021})
\end{bbook}
\endbibitem

\bibitem[\protect\citeauthoryear{Blanes et~al.}{2014}]{Blanes_Casas_Sanz_Serna:2014}
\begin{barticle}
\bauthor{\bsnm{Blanes}, \binits{S.}},
\bauthor{\bsnm{Casas}, \binits{F.}},
\bauthor{\bsnm{Sanz-Serna}, \binits{J.M.}}:
\batitle{{Numerical Integrators for the Hybrid Monte Carlo Method}}.
\bjtitle{SIAM Journal on Scientific Computing}
\bvolume{36}(\bissue{4}),
\bfpage{1556}--\blpage{1580}
(\byear{2014})
\end{barticle}
\endbibitem

\bibitem[\protect\citeauthoryear{Nagar et~al.}{2024}]{Nagar_Pendas_Sanz_Serna_Akhmatskaya:2023}
\begin{barticle}
\bauthor{\bsnm{Nagar}, \binits{L.}},
\bauthor{\bsnm{Fernández-Pendás}, \binits{M.}},
\bauthor{\bsnm{Sanz-Serna}, \binits{J.M.}},
\bauthor{\bsnm{Akhmatskaya}, \binits{E.}}:
\batitle{{Adaptive multi-stage integration schemes for Hamiltonian Monte Carlo}}.
\bjtitle{J.~Comp.\ Phys.}
\bvolume{502},
\bfpage{112800}
(\byear{2024})
\end{barticle}
\endbibitem

\bibitem[\protect\citeauthoryear{Casas et~al.}{2022}]{Casas_Sanz_Serna_Shaw:2022}
\begin{botherref}
\oauthor{\bsnm{Casas}, \binits{F.}},
\oauthor{\bsnm{Sanz-Serna}, \binits{J.M.}},
\oauthor{\bsnm{Shaw}, \binits{L.}}:
{Split Hamiltonian Monte Carlo revisited}.
Stat Comput
\textbf{32}
(2022)
\end{botherref}
\endbibitem

\bibitem[\protect\citeauthoryear{Blanes et~al.}{2021}]{Blanes_Calvo_Casas_Sanz_Serna:2021}
\begin{barticle}
\bauthor{\bsnm{Blanes}, \binits{S.}},
\bauthor{\bsnm{Calvo}, \binits{M.P.}},
\bauthor{\bsnm{Casas}, \binits{F.}},
\bauthor{\bsnm{Sanz-Serna}, \binits{J.M.}}:
\batitle{{Symmetrically Processed Splitting Integrators for Enhanced Hamiltonian Monte Carlo Sampling}}.
\bjtitle{SIAM Journal on Scientific Computing}
\bvolume{43}(\bissue{5}),
\bfpage{3357}--\blpage{3371}
(\byear{2021})
\end{barticle}
\endbibitem

\bibitem[\protect\citeauthoryear{Faou et~al.}{2009}]{Faou_Lubich:2009}
\begin{barticle}
\bauthor{\bsnm{Faou}, \binits{E.}},
\bauthor{\bsnm{Gradinaru}, \binits{V.}},
\bauthor{\bsnm{Lubich}, \binits{C.}}:
\batitle{{Computing Semiclassical Quantum Dynamics with Hagedorn Wavepackets}}.
\bjtitle{SIAM Journal on Scientific Computing}
\bvolume{31}(\bissue{4}),
\bfpage{3027}--\blpage{3041}
(\byear{2009})
\end{barticle}
\endbibitem

\bibitem[\protect\citeauthoryear{Meyer et~al.}{1990}]{Meyer_Cederbaum:1990}
\begin{barticle}
\bauthor{\bsnm{Meyer}, \binits{H.-D.}},
\bauthor{\bsnm{Manthe}, \binits{U.}},
\bauthor{\bsnm{Cederbaum}, \binits{L.S.}}:
\batitle{{The Multi-configurational Time-dependent {Hartree} Approach}}.
\bjtitle{Chem.\ Phys.\ Lett.}
\bvolume{165}(\bissue{1}),
\bfpage{73}--\blpage{78}
(\byear{1990})
\end{barticle}
\endbibitem

\bibitem[\protect\citeauthoryear{Raab and Meyer}{2000}]{Raab_Meyer:2000}
\begin{barticle}
\bauthor{\bsnm{Raab}, \binits{A.}},
\bauthor{\bsnm{Meyer}, \binits{H.-D.}}:
\batitle{{A numerical study on the performance of the multiconfiguration time-dependent Hartree method for density operators}}.
\bjtitle{J.~Chem.\ Phys.}
\bvolume{112},
\bfpage{10718}--\blpage{10729}
(\byear{2000})
\end{barticle}
\endbibitem

\bibitem[\protect\citeauthoryear{Kucar et~al.}{1987}]{Kucar_Cederbaum:1987}
\begin{barticle}
\bauthor{\bsnm{Kucar}, \binits{J.}},
\bauthor{\bsnm{Meyer}, \binits{H.-D.}},
\bauthor{\bsnm{Cederbaum}, \binits{L.S.}}:
\batitle{{Time-dependent rotated Hartree approach}}.
\bjtitle{Chem.\ Phys.\ Lett.}
\bvolume{140}(\bissue{5}),
\bfpage{525}--\blpage{530}
(\byear{1987})
\end{barticle}
\endbibitem

\bibitem[\protect\citeauthoryear{Lasser and R\"{o}blitz}{2010}]{Lasser_Roeblitz:2010}
\begin{barticle}
\bauthor{\bsnm{Lasser}, \binits{C.}},
\bauthor{\bsnm{R\"{o}blitz}, \binits{S.}}:
\batitle{{Computing Expectation Values for Molecular Quantum Dynamics}}.
\bjtitle{SIAM Journal on Scientific Computing}
\bvolume{32}(\bissue{3}),
\bfpage{1465}--\blpage{1483}
(\byear{2010})
\end{barticle}
\endbibitem

\bibitem[\protect\citeauthoryear{{Henon} and {Heiles}}{1964}]{Henon_Heiles:1964}
\begin{barticle}
\bauthor{\bsnm{{Henon}}, \binits{M.}},
\bauthor{\bsnm{{Heiles}}, \binits{C.}}:
\batitle{{The applicability of the third integral of motion: Some numerical experiments}}.
\bjtitle{Astronomical Journal}
\bvolume{69},
\bfpage{73}
(\byear{1964})
\end{barticle}
\endbibitem

\bibitem[\protect\citeauthoryear{Kay}{1989}]{Kay:1989}
\begin{barticle}
\bauthor{\bsnm{Kay}, \binits{K.G.}}:
\batitle{{The matrix singularity problem in the time-dependent variational method}}.
\bjtitle{Chem.\ Phys.}
\bvolume{137}(\bissue{1--3}),
\bfpage{165}--\blpage{175}
(\byear{1989})
\end{barticle}
\endbibitem

\bibitem[\protect\citeauthoryear{Choi and Van{\'{i}}{\v{c}}ek}{2019}]{Choi_Vanicek:2019a}
\begin{barticle}
\bauthor{\bsnm{Choi}, \binits{S.}},
\bauthor{\bsnm{Van{\'{i}}{\v{c}}ek}, \binits{J.}}:
\batitle{{{A time-reversible integrator for the time-dependent Schr{\"{o}}dinger equation on an adaptive grid}}}.
\bjtitle{J.~Chem.\ Phys.}
\bvolume{151},
\bfpage{234102}
(\byear{2019})
\end{barticle}
\endbibitem

\bibitem[\protect\citeauthoryear{Kaledin and Miller}{2003}]{Kaledin_Miller:2003}
\begin{barticle}
\bauthor{\bsnm{Kaledin}, \binits{A.L.}},
\bauthor{\bsnm{Miller}, \binits{W.H.}}:
\batitle{{Time averaging the semiclassical initial value representation for the calculation of vibrational energy levels}}.
\bjtitle{J.~Chem.\ Phys.}
\bvolume{118}(\bissue{16}),
\bfpage{7174}--\blpage{7182}
(\byear{2003})
\end{barticle}
\endbibitem

\bibitem[\protect\citeauthoryear{Buchholz et~al.}{2016}]{Buchholz_Ceotto:2016}
\begin{barticle}
\bauthor{\bsnm{Buchholz}, \binits{M.}},
\bauthor{\bsnm{Grossmann}, \binits{F.}},
\bauthor{\bsnm{Ceotto}, \binits{M.}}:
\batitle{{Mixed semiclassical initial value representation time-averaging propagator for spectroscopic calculations}}.
\bjtitle{J.~Chem.\ Phys.}
\bvolume{144}(\bissue{9}),
\bfpage{094102}
(\byear{2016})
\end{barticle}
\endbibitem

\bibitem[\protect\citeauthoryear{Buchholz et~al.}{2018}]{Buchholz_Ceotto:2018}
\begin{barticle}
\bauthor{\bsnm{Buchholz}, \binits{M.}},
\bauthor{\bsnm{Grossmann}, \binits{F.}},
\bauthor{\bsnm{Ceotto}, \binits{M.}}:
\batitle{{Simplified approach to the mixed time-averaging semiclassical initial value representation for the calculation of dense vibrational spectra}}.
\bjtitle{J.~Chem.\ Phys.}
\bvolume{148}(\bissue{11}),
\bfpage{114107}
(\byear{2018})
\end{barticle}
\endbibitem

\bibitem[\protect\citeauthoryear{Filinov}{1986}]{Filinov:1986}
\begin{barticle}
\bauthor{\bsnm{Filinov}, \binits{V.S.}}:
\batitle{{Calculation of the Feynman integrals by means of the Monte Carlo method}}.
\bjtitle{Nucl.\ Phys.~B}
\bvolume{271}(\bissue{3-4}),
\bfpage{717}--\blpage{725}
(\byear{1986})
\end{barticle}
\endbibitem

\bibitem[\protect\citeauthoryear{Makri and Miller}{1987}]{Makri_Miller:1987}
\begin{barticle}
\bauthor{\bsnm{Makri}, \binits{N.}},
\bauthor{\bsnm{Miller}, \binits{W.H.}}:
\batitle{{Monte Carlo integration with oscillatory integrands: implications for Feynman path integration in real time}}.
\bjtitle{Chem.\ Phys.\ Lett.}
\bvolume{139},
\bfpage{10}--\blpage{14}
(\byear{1987})
\end{barticle}
\endbibitem

\bibitem[\protect\citeauthoryear{Grossmann}{2006}]{Grossmann:2006}
\begin{barticle}
\bauthor{\bsnm{Grossmann}, \binits{F.}}:
\batitle{{A Semiclassical Hybrid Approach to Many Particle Quantum Dynamics}}.
\bjtitle{J.~Chem.\ Phys.}
\bvolume{125}(\bissue{1}),
\bfpage{014111}
(\byear{2006})
\end{barticle}
\endbibitem

\bibitem[\protect\citeauthoryear{Goletz and Grossmann}{2009}]{Goletz_Grossmann:2009}
\begin{barticle}
\bauthor{\bsnm{Goletz}, \binits{C.M.}},
\bauthor{\bsnm{Grossmann}, \binits{F.}}:
\batitle{{Decoherence and dissipation in a molecular system coupled to an environment: An application of semiclassical hybrid dynamics}}.
\bjtitle{J.~Chem.\ Phys.}
\bvolume{130}(\bissue{24}),
\bfpage{244107}
(\byear{2009})
\end{barticle}
\endbibitem

\bibitem[\protect\citeauthoryear{Papoulis and Pillai}{2002}]{Papoulis_Pillai:2002}
\begin{bbook}
\bauthor{\bsnm{Papoulis}, \binits{A.}},
\bauthor{\bsnm{Pillai}, \binits{S.U.}}:
\bbtitle{{Probability, Random Variables, and Stochastic Processes}}.
\bpublisher{McGraw-Hill},
\blocation{New York}
(\byear{2002})
\end{bbook}
\endbibitem

\bibitem[\protect\citeauthoryear{Bou{-}Rabee and Sanz{-}Serna}{2013}]{Bou_Sanz:2013}
\begin{bbook}
\bauthor{\bsnm{Bou{-}Rabee}, \binits{N.}},
\bauthor{\bsnm{Sanz{-}Serna}, \binits{J.M.}}:
\bbtitle{{Geometric Integrators and the Hamiltonian Monte Carlo Method}}.
\bpublisher{Cambridge University Press},
\blocation{Cambridge}
(\byear{2013})
\end{bbook}
\endbibitem

\end{thebibliography}

\appendix

\section{Bound for intrinsic measures}\label{Appendix:Bound_intrinsic_measure}
When an exact solution is not available, intrinsic measures can be used as indicators of the goodness of an approximation.
Assume we have an unbiased estimator $A_N(t) = \frac{1}{N}\sum_j h_t(w_j)$ of $\langle \hat{A} \rangle_t$ with existing second moment. 
One can consider the distance between $A_N$ and $A_{2N}$ since
\begin{align}
\begin{split}
    &\mathbb{E}[\vert A_N - A_{2N} \vert^2] 
    \\& = \mathbb{E}\left[ (\vert A_N - \langle \hat{A} \rangle_t) + (\langle \hat{A} \rangle_t - A_{2N}) \vert^2 \right]
    \\& = \mathbb{E}[\vert A_N - \langle \hat{A} \rangle_t \vert^2] + \mathbb{E}[\vert A_{2N}- \langle \hat{A} \rangle_t \vert^2] - 2\mathrm{Re}\left(\mathbb{E}[A_{N} - \langle \hat{A} \rangle_t]^* \mathbb{E}[ A_{2N} - \langle \hat{A} \rangle_t ]\right) 
    \\&= \frac{\mathbb{V}_t}{N}  + \frac{\mathbb{V}_t}{2N} - 0 
    \\& = \frac{3}{2N}\mathbb{V}_t,
\end{split}
\end{align}
where we assumed that the samples of $A_N$ and $A_{2N}$ were independent. Let us recall the formulas \eqref{EQ:Variance_sqrtH} and \eqref{EQ:Variance_opt} for the variances
\begin{align}
    \mathbb{V}_t[\hat{A}, \psi_0, \rho^{\rm dbl}_{\textup{sqrt-H}} ] = \kappa_{\textup{sqrt-H}}^2 \int_{\mathbb{R}^{4D}} \vert f_0(w) \vert \vert \Phi_t(w) \vert^2 \vert O_t[\hat{A}](w) \vert^2 \,dw - \vert\langle \hat{A} \rangle_t\vert^2
\end{align}
and
\begin{align}
    \mathbb{V}_t[\hat{A}, \psi_0, \rho_{\rm opt} ] = \kappa_{\rm opt} \int_{\mathbb{R}^{4D}} \vert f_0(w) \vert \vert \Phi_t(w) \vert^2 \frac{\vert O_t[\hat{A}](w) \vert^2}{\vert O_0[\hat{A}](w) \vert} \,dw - \vert\langle \hat{A} \rangle_t\vert^2.
\end{align}
In the calculation, we can simultaneously evaluate the appearing integrals without any additional work using another Monte Carlo estimator. 
Assuming we know $\kappa_{\textup{sqrt-H}}$, we then have
\begin{align}
    \int_{\mathbb{R}^{4D}} \vert f_0(w) \vert \vert \Phi_t(w) \vert^2 \vert O_t[\hat{A}](w) \vert^2 \,dw \approx \frac{\kappa^2_{\textup{sqrt-H}}}{N}\sum_{j=1}^N  \vert \Phi_t(y_j,z_j) \vert^2 \vert O_t[\hat{A}](y_j,z_j) \vert^2,
\end{align}
for all operators $\hat{A}$ and with $y_j,z_j \sim \rho_{\textup{sqrt-H}}$ independent.
For the optimal approach, whenever $\kappa_{\rm opt}$ is known, we can use
\begin{align}
    \int_{\mathbb{R}^{4D}} \vert f_0(w) \vert \vert \Phi_t(w) \vert^2 \frac{\vert O_t[\hat{A}](w) \vert^2}{\vert O_0[\hat{A}](w) \vert} \,dw \approx \frac{\kappa_{\rm opt}}{N}\sum_{j=1}^N  \vert \Phi_t(w_j) \vert^2 \frac{\vert O_t[\hat{A}](w_j) \vert^2}{\vert O_0[\hat{A}](w_j) \vert^2},
\end{align}
where $w_j\sim \rho_{\rm opt}$ are independent.
If $\kappa_{\rm opt}$ is unknown, a slight change in the weighted importance sampling estimator leads to
\begin{align}
\begin{split}
    &\kappa_{\rm opt}\int_{\mathbb{R}^{4D}} \vert f_0(w) \vert \vert \Phi_t(w) \vert^2 \frac{\vert O_t[\hat{A}](w) \vert^2}{\vert O_0[\hat{A}](w) \vert} \,dw
    \\&\approx \frac{\frac{1}{N}\sum_{j=1}^N  W(w_j)\vert f_0(w_j)\vert\vert \Phi_t(w_j) \vert^2 \vert O_t[\hat{A}](w_j) \vert^2/(\rho^{\rm dbl}_{\rm H}(w_j))\vert O_0[\hat{A}](w_j) \vert)}{\frac{1}{N^2}\left(\sum_{j=1}^N W(w_j)\right)^2},
\end{split}
\end{align}
with the weight $W=\rho^{\rm dbl}_{\rm H }/\rho_{\rm opt}$ and independent samples $w_j \sim \rho_{\rm opt}$.

\section{Moments of Gaussian random variables} \label{Appendix:Moments_of_Gaussians}
Consider a multivariate Gaussian
\begin{align}
    G^{\rm M}(x)= (2\pi\epsilon)^{-D/2} \det{\Xi}^{-1/2} \exp{\left( -\frac{1}{2\epsilon} (x-\mu)^T \Xi^{-1} (x-\mu) \right)}
\end{align}
with mean vector $\mu\in\mathbb{R}^D$ and symmetric, positive-definite covariance matrix $\Xi\in\mathbb{R}^{D \times D}$.
Let $\Xi = L L^T$ be its Cholesky decomposition with a lower triangular matrix $L\in\mathbb{R}^{D \times D}$ and let $Z\sim N(\mu, \Xi).$ 
Moreover, let $\textrm{Pol}: \mathbb{R}^D \to \mathbb{R}$ be a multivariate polynomial. Then,
\begin{align}
    \begin{split}
        \mathbb{E}[\textrm{Pol}(Z)]  = \mathbb{E}[\textrm{Pol}(LX+\mu)]
    \end{split}
\end{align}
with the $D$-dimensional random vector $X=(X_1, \dots, X_D) \in \mathbb{R}^D$, where $X_1,\dots, X_D$ are independent and identically distributed with respect to $N(0,\epsilon)$. Since the composition of two polynomials is a polynomial and since $X_1, \dots, X_D$ are independent and identically distributed, $\mathbb{E}[\textrm{Pol}(LX+\mu)]$ depends polynomially on $\epsilon$, $\mu$ and $\Xi$ \cite[Chapter 5.4]{Papoulis_Pillai:2002}.

\section{Formulas for inner products} \label{Appendix:Inner_products}
In the following, we provide analytical formulas for the overlaps of the form
    $\langle g_y , \hat{A}  g_z \rangle$
for important examples of $\hat{A}$ 
 with a Gaussian wavepacket
\begin{align}
  g_z(x)= \left( \frac{\det \Gamma }{\pi^D \epsilon^D} \right)^{1/4} \exp{\left[-\frac{1}{2\epsilon} (x-q)^T  \Gamma  (x-q) +\frac{i}{\epsilon}p^T (x-q) \right] }, \quad z=(q,p). 
\end{align}

\begin{example}[$\hat{A}=\mathrm{Id}$]\label{Ex:A1}

Using the substitution $x \mapsto x +(q_y+q_z)/2$ and the Fourier transform, we obtain
\begin{align}
\begin{split}
    \langle g_y  , g_z \rangle  
     & = \left( \frac{\det \Gamma }{\pi^D \epsilon^D} \right)^{1/2} \int_{\mathbb{R}^{D}} \exp{\left[-\frac{1}{2\epsilon} (x-q_y/2+q_z/2)^T  \Gamma  (x-q_y/2+q_z/2)\right]}
     \\& \times \exp{\left[-\frac{1}{2\epsilon} (x-q_z/2+q_y/2)^T  \Gamma  (x-q_z/2+q_y/2)\right]}
     \\ & \times \exp{\left[-\frac{i}{\epsilon} p_y^T   (x-q_y/2+q_z/2) +\frac{i}{\epsilon}p_z^T  (x-q_z/2+q_y/2) \right]} \,dx
    \\& = \exp{ \left\{ \left[ -\frac{1}{4} (y-z)^T  \Sigma_0  (y-z) + \frac{i}{2} (p_y+p_z)^T  (q_y-q_z) \right]/\epsilon \right\}},
\end{split}
\end{align}
with $\Sigma_0=\mathrm{diag}(\Gamma, \Gamma^{-1})\in\mathbb{R}^{2D}$.
\end{example}
\begin{example}[$\hat{A}=\hat{q}_j$] \label{Ex:A2}
Similarly, for the position operator it follows
\begin{align}
\begin{split}
\langle g_y  , \hat{q}_j  g_z \rangle & 
     = \frac{1}{2}\left( q_y +q_z -i \Gamma^{-1}  p_y+ i \Gamma^{-1}  p_z \right)_j \langle g_y , g_z \rangle.
\end{split}
\end{align}
 
\end{example}

\begin{example}[$\hat{A}=\hat{q}^2_j$]
For the second moment of the position operator it holds
\begin{align}
\begin{split}
\langle g_y  , \hat{q}_j^2  g_z \rangle 
       = \left[ 2\epsilon\Gamma^{-1} + (q_y+q_z-i\Gamma^{-1}p_y + i\Gamma^{-1}p_z)^2\right]_j \frac{\langle g_y , g_z \rangle}{4}.
\end{split}
\end{align}

\end{example}

\begin{example}[$\hat{A} = \hat{p}_j$]
Changing to the momentum representation and applying the same techniques as before, we obtain
\begin{align}
    \begin{split}
        \langle g_y , \hat{p}_j g_z \rangle = \frac{1}{2}\left( p_y+p_z+i\Gamma q_y - i \Gamma q_z \right)_j \langle g_y , g_z \rangle.
    \end{split}
\end{align}

\end{example}

\begin{example}[$\hat{A} = \hat{p}^2_j$]
For the kinetic energy it holds
\begin{align}
\langle g_y , \hat{p}^2_j  g_z \rangle = \left[ 2\epsilon \Gamma + (p_y+p_z+i\Gamma q_y - i\Gamma q_z)^2\right]_j \frac{\langle g_y , g_z \rangle}{4}.
\end{align}

\end{example}

\begin{example}[Henon-Heiles potential \eqref{Henon-Heiles_Potential}]
By linearity, the overlap is given by
\begin{align}
\begin{split}
    \langle g_y , V g_z \rangle = \underbrace{\frac{1}{2} \sum_{j=1}^D  \langle g_y, \hat{q}_j^2 g_z \rangle}_{=: I_1} + \underbrace{\sigma \sum_{j=1}^{D-1} \langle g_y,  \left( \hat{q}_j \hat{q}_{j+1}^2 - \frac{\hat{q}_j^3}{3} \right) g_z \rangle}_{=:I_2} + \underbrace{\frac{\sigma^2}{16} \sum_{j=1}^{D-1} \langle g_y,   \left( \hat{q}_j^2 + \hat{q}_{j+1}^2 \right)^2 g_z \rangle}_{=:I_3}.
     \end{split}
\end{align}
We obtain
\begin{align}
    \begin{split}
        I_1 = \frac{\langle g_y , g_z \rangle}{8} \sum_{j=1}^D \left[ 2\epsilon\Gamma^{-1} + (\overline{q}+i\Gamma^{-1}\overline{p})^2\right]_j ,
    \end{split}
\end{align}
\begin{align}
    \begin{split}
        I_2 = \sigma \langle g_y , g_z \rangle \sum_{j=1}^{D-1} & \left\{\frac{ (\overline{q}+i\Gamma^{-1}\overline{p})_j}{2}
        \frac{\left[ 2\epsilon\Gamma^{-1} + (\overline{q}+i\Gamma^{-1}\overline{p})^2\right]_{j+1}}{4}\right.
        \\& \left.-\frac{\left[6\epsilon\Gamma^{-1}\overline{q} + 6i\epsilon\Gamma^{-2} \overline{p} + (\overline{q} + i\Gamma^{-1}\overline{p})^3 \right]_j}{24}\right\}
    \end{split}
\end{align}
and
\begin{align}
    \begin{split}
        I_3 
        &= \frac{\sigma^2 \langle g_y , g_z \rangle}{16} \sum_{j=1}^{D-1} \left\{\frac{\left[ 2\epsilon\Gamma^{-1} + (\overline{q}+i\Gamma^{-1}\overline{p})^2\right]_j\left[ 2\epsilon\Gamma^{-1} + (\overline{q}+i\Gamma^{-1}\overline{p})^2\right]_{j+1}}{8} \right.
        \\& \left.+ \frac{\left[12\epsilon \Gamma^{-1} \overline{q}^2+24ia^3\overline{q}s-12a^4s^2+12\epsilon^2\Gamma^{-2} + (\overline{q} +i\Gamma^{-1}\overline{p})^4 \right]_j}{16} \right.
        \\& \left.+ \frac{\left[12\epsilon\Gamma^{-1}\overline{q}^2+24ia^3\overline{q}s-12a^4s^2+12\epsilon^2\Gamma^{-2} + (\overline{q} +i\Gamma^{-1}\overline{p})^4 \right]_{j+1}}{16}  
         \right\} 
    \end{split}
\end{align}
with the abbreviations $\overline{q} := q_y+q_z$, $\overline{p}  := p_z - p_y$,  $a:= \epsilon^{1/2}\Gamma^{-1/2}$ and $s:= (\epsilon \Gamma)^{-1/2} \overline{p}$. 
\end{example}

\section{Variances for the square root Husimi approach} \label{Appendix:5}
We provide analytical formulas for the variance \eqref{EQ:Variance_sqrtH} with a Gaussian initial state $\psi_0 = g_{z_0}$. For notational simplicity, we assume $\Gamma=\mathrm{Id}_D$ and employ the notations $w=(y,z)$ and $w_0=(z_0,z_0).$ To obtain the formulas for general $\Gamma$, one may invoke the transformation 
\begin{align}
    (w-w_0) \mapsto \begin{pmatrix}
      \Sigma_0^{-1/2}&0\\0 &  \Sigma_0^{-1/2} 
    \end{pmatrix} (w-w_0),
\end{align}
where $\Sigma_0=\text{diag}{(\Gamma,\Gamma^{-1})}\in\mathbb{R}^{2D\times 2D}$. Moreover, whenever we consider $4D\times 4D$ matrices of the form $\begin{pmatrix}
    a&b\\c&d
\end{pmatrix}$, $a,b,c,d\in\mathbb{R}$, we understand $\begin{pmatrix}
    a\mathrm{Id}_{2D}&b\mathrm{Id}_{2D}\\c\mathrm{Id}_{2D}&d\mathrm{Id}_{2D}
\end{pmatrix}$.
\subsection{Initial time}\label{Variances_Initial_Time}
\begin{example}[$\hat{A} = \mathrm{Id}$]
    \begin{align}
\begin{split}
   &\mathbb{V}_0[ \textup{Id} , g_{z_0}, \rho^{\rm dbl}_{\textup{sqrt-H}} ] 
    \\& = (\pi\epsilon)^{-2D} \int_{\mathbb{R}^{4D}} 
     \exp{\left[ -\frac{1}{4\epsilon} (w-w_0)^T  \begin{pmatrix}
        3 & -2\\ -2 & 3 
    \end{pmatrix}  (w-w_0) \right]} \,dw -1
    \\ & =  \left(\frac{16}{5}\right)^D-1.
\end{split}
\end{align}
\end{example}
\begin{example}[$\hat{A} = \hat{q}_j$]\label{Ex:A3}
Employing the transformation \begin{align}
     \begin{pmatrix}
        z-z_0\\y-z_0
    \end{pmatrix} \mapsto \frac{1}{2} \begin{pmatrix}
       1&-1\\1 & 1 
    \end{pmatrix}  \begin{pmatrix}
        z-z_0 \\ y-z_0
    \end{pmatrix},
\end{align} we obtain 
\begin{align}
\begin{split}
   \mathbb{V}_0[\hat{q}_j , g_{z_0}, \rho^{\rm dbl}_{\textup{sqrt-H}}] 
    & = \frac{(\pi\epsilon)^{-2D}}{4} \int_{\mathbb{R}^{4D}} \vert q_{z,j}+q_{y,j}+i p_{z,j} - i p_{y,j} \vert^2
    \\ & \times \exp{\left[ -\frac{1}{4\epsilon} \begin{pmatrix}
        z-z_0 \\ y-z_0
    \end{pmatrix}^T  \begin{pmatrix}
        3 & -2 \\ -2 & 3 
    \end{pmatrix}  \begin{pmatrix}
        z-z_0 \\ y-z_0
    \end{pmatrix} \right]} \,d(z,y) -q_{0,j}^2
    \\ & =  \frac{1}{4} \left(\frac{16}{5}\right)^D\left[ 4q_{0,j}^2 + \frac{24}{5}\epsilon  \right]-q_{0,j}^2
\end{split}
\end{align}
where the $j-$th component of vectors is denoted by a subscript.

\end{example}
\begin{example}[$\hat{A} = \hat{p}_j$]
    Similarly to the position operator, for the momentum we obtain
\begin{align}
\begin{split}
   &\mathbb{V}_0[\hat{p}_j , g_{z_0}, \rho^{\rm dbl}_{\textup{sqrt-H}}] 
    =  \frac{1}{4} \left(\frac{16}{5}\right)^D\left[ 4p_{0,j}^2 + \frac{24}{5}\epsilon\  \right]-p_{0,j}^2.
\end{split}
\end{align} 
\end{example}
\begin{example}[$\hat{A} = \sum_{j=1}^{D}\hat{q}_j^2/2$] \label{Ex:16}
For simplicity, we assume the initial position $q_0=c_1(1,\dots, 1)$ and initial momentum $p_0=c_2(1, \dots, 1)$ for some constants $c_1,c_2\in\mathbb{R}$, i.\,e. the same parameters in each dimension. Then, the variance is given by
\begin{align}
\begin{split}
   &\mathbb{V}_0[\hat{A} , g_{z_0}, \rho^{\rm dbl}_{\textup{sqrt-H}}] 
     \\& = \frac{1}{16} \left( \frac{16}{5}\right)^D \left[ D^2 \left( 2c_1^2 + \frac{13    }{5}\epsilon \right)^2 + D \frac{96}{25}\left( 3\epsilon + 5c_1^2\right) \right]- \frac{D^2}{4}(\epsilon+2c_1^2)^2.
\end{split}
\end{align}
\end{example}
\begin{example}[$\hat{A} =  \sum_{j=1}^{D}\hat{p}_j^2/2$]
Similarly to the harmonic potential, for the kinetic energy we obtain
\begin{align}
\begin{split}
   &\mathbb{V}_0[\hat{A} , g_{z_0}, \rho^{\rm dbl}_{\textup{sqrt-H}}] 
    \\& = \frac{1}{16} \left( \frac{16}{5}\right)^D \left[ D^2 \left( 2c_2^2 + \frac{13    }{5}\epsilon \right)^2 + D \frac{96}{25}\left( 3\epsilon + 5c_2^2\right) \right]- \frac{D^2}{4}(\epsilon+2c_2^2)^2.
\end{split}
\end{align} 
\end{example}
\begin{example}[$\hat{A} =  \sum_{j=1}^{D}(\hat{p}_j^2 + \hat{q}_j^2)/2 $]
\begin{align}
\begin{split}
   &\mathbb{V}_0[\hat{A} , g_{z_0}, \rho^{\rm dbl}_{\textup{sqrt-H}}] 
    \\& = \frac{1}{4} \left( \frac{16}{5}\right)^D \left[ D^2  \left( (c_1^2 + c_2^2) + \frac{13    }{5}\epsilon \right)^2 + D \frac{24}{5}\left( \frac{6}{5}\epsilon + c_1^2+c_2^2\right) \right]
    \\&- D^2(\epsilon+c_1^2+c_2^2)^2.
\end{split}
\end{align} 
\end{example}

\subsection{Harmonic potential}\label{Variances_Harmonic_Potential}
We assume a Gaussian initial state with $\Gamma = \textup{Id}_D$. For a harmonic potential $V(x)=\vert x\vert^2/2$, there exist explicit solutions to the position, momentum, and Herman--Kluk prefactor. In particular, for all $z=(q,p)\in \mathbb{R}^{2D}$,
\begin{align}
    &q(t)=q \cos(t)+p \sin(t),
    \\& p(t)=p\cos(t)-q\sin(t),
    \\& \vert R_t(z)\vert^2=1.
\end{align}
Moreover, the Herman--Kluk approximation itself is exact. Hence, we can calculate the behaviour of the variances over time.

\begin{example}[$\hat{A}=\mathrm{Id}$]
\begin{align}
    \begin{split}
   &\mathbb{V}_t[\textup{Id} , g_{z_0}, \rho^{\rm dbl}_{\textup{sqrt-H}}] 
     = \left( \frac{16}{5}\right)^D-1.
    \end{split}
\end{align}
\end{example}

\begin{example}[$\hat{A}=\hat{q}_j$]
    \begin{align}
    \begin{split}
   &\mathbb{V}_t[\hat{q}_j, g_{z_0}, \rho^{\rm dbl}_{\textup{sqrt-H}}] = \frac{1}{4}\left(\frac{16}{5} \right)^D \left[ 4q_{0,j}(t)^2+\frac{24}{5}\epsilon \right] - q_{0,j}(t).
    \end{split}
\end{align}

\end{example}

\begin{example}[$\hat{A}=\hat{p}_j$]
    \begin{align}
    \begin{split}
   &\mathbb{V}_t[\hat{p}_j, g_{z_0}, \rho^{\rm dbl}_{\textup{sqrt-H}}] = \frac{1}{4}\left(\frac{16}{5} \right)^D \left[ 4p_{0,j}(t)^2+\frac{24}{5}\epsilon \right] - p_{0,j}(t).
    \end{split}
\end{align}

\end{example}

\begin{example}[$\hat{A}=\sum_{j=1}^D\hat{q}^2_j/2$] Under the same assumption as in \Cref{Ex:16}, i.\,e. $q_0=c_1(0)(1, \dots ,1) $ and $p_0=c_2(0)(1, \dots ,1)$, $c_1(0),c_2(0) \in \mathbb{R}$, let
$c_1(t) = c_1(0)\cos(t)+c_2(0)\sin(t)$
and $c_2(t) = c_1(0)\cos(t)-c_2(0)\sin(t)$.
Then
    \begin{align}
\begin{split}
   \mathbb{V}_t[\hat{A} , g_{z_0}, \rho^{\rm dbl}_{\textup{sqrt-H}}] 
     & = \frac{1}{16} \left( \frac{16}{5}\right)^D \left[ D^2 \left( 2c_1(t)^2 + \frac{13    }{5}\epsilon \right)^2 + D \frac{96}{25}\left( 3\epsilon + 5c_1(t)^2\right) \right]
     \\&- \frac{D^2}{4}(\epsilon+2c_1(t)^2)^2.
\end{split}
\end{align}

\end{example}

\begin{example}[$\hat{A}=\sum_{j=1}^D\hat{p}^2_j/2$]
Similarly to the harmonic potential, it holds
    \begin{align}
    \begin{split}
    \mathbb{V}_t[\hat{A} , g_{z_0}, \rho^{\rm dbl}_{\textup{sqrt-H}}] 
    & = \frac{1}{16} \left( \frac{16}{5}\right)^D \left[ D^2 \left( 2c_2(t)^2 + \frac{13    }{5}\epsilon \right)^2 + D \frac{96}{25}\left( 3\epsilon + 5c_2(t)^2\right) \right]
    \\&- \frac{D^2}{4}(\epsilon+2c_2(t)^2)^2.
\end{split}
\end{align} 

\end{example}

\begin{example}[$\hat{A}=\sum_{j=1}^D(\hat{q}^2_j+\hat{p}^2_j)/2$]
    \begin{align}
\begin{split}
   &\mathbb{V}_t[\hat{A} , g_{z_0}, \rho^{\rm dbl}_{\textup{sqrt-H}}] 
    \\& = \frac{1}{4} \left( \frac{16}{5}\right)^D \left[ D^2  \left( (c_1(t)^2 + c_2(t)^2) + \frac{13    }{5}\epsilon \right)^2 + D \frac{24}{5}\left( \frac{6}{5}\epsilon + c_1(t)^2+c_2(t)^2\right) \right]
    \\&- D^2(\epsilon+c_1(t)^2+c_2(t)^2)^2.
\end{split}
\end{align}

\end{example}

\section{Hamiltonian Monte Carlo algorithm}\label{Appendix:2}
The Hamiltonian Monte Carlo algorithm \cite{Bou_Sanz:2013} belongs to the class of Metropolis--Hastings algorithms whose aim is to generate samples of probability densities $\rho$ where sampling directly from $\rho$ on $\mathbb{R}^{D}$ is not viable either because the expression of $\rho$ is too complicated or because it is only known up to a multiplicative constant. 
Using a fictitious momentum variable, it generates a Markov chain $\{\xi^0, \xi^1, \dots \} = \{\xi^n\}_{n\in \mathbb{N}}$ which has $\rho$ as its invariant density. 
Introducing the potential energy $U(\xi)= - \log(\rho(\xi))$, the kinetic energy $K(\Xi)=\Xi^T  M^{-1}  \Xi/2$ with mass matrix $M$, and the Hamiltonian $H(\xi,\Xi) = U(\xi) + K(\Xi)$, one simulates a Markov chain according to \Cref{Alg:HMC}.
\begin{algorithm}
\caption{Hamiltonian Monte Carlo}
\label{Alg:HMC}
\begin{algorithmic}
\State Choose an initial double phase-space coordinate $\zeta^0\in\mathbb{R}^{4D}$. Then, for $n=0,\dots, N-1$ do:
\begin{enumerate}
    \item Generate a random momentum variable $\Xi^n$ from a multivariate normal distribution with a zero mean and a covariance matrix identical to the mass matrix $M$.
    \item Using the introduced potential and kinetic energy, propagate $(\zeta^n,\Xi^n)$ in time using a symplectic and time-reversible integrator to obtain a new phase-space point $(\zeta^*,\Xi^*)$.
    \item Set $\zeta^{n+1}=\zeta^*$ with probability $\alpha$, where
    \begin{align}
        \alpha = \min \left\{1, \exp\left[ H(\zeta^n, \Xi^n) - H(\zeta^*, \Xi^*) \right]  \right\},
    \end{align}
    otherwise set $\zeta^{n+1}=\zeta^n$.
\end{enumerate} 
\end{algorithmic}
\end{algorithm}
This algorithm has a few parameters, namely, the initial point $\zeta_0$ and the mass matrix $M$ as well as the step size $\Delta t$ and final time $T$ for the numerical integrator. 
These must be chosen carefully to obtain a valuable sequence $\{\zeta^n\}$. 

\subsection{Gradients for Hamiltonian Monte Carlo}
In the following, we provide a method to derive gradients for the Hamiltonian Monte Carlo algorithm to generate a Markov chain distributed with respect to $\rho_{\rm opt}$ for specific examples of $\hat{A}$. We assume a Gaussian initial state $\psi_0 = g_{z_0}$, $z_0\in\mathbb{R}^{2D}$, and a unit width matrix $\Gamma=\mathrm{Id}_D$. In our examples, we have the polynomial decomposition
\begin{align}
\begin{split}\label{AppendixE:pol}
    \rho_{\rm opt}(w)&\propto\vert\langle g_z , \psi_0 \rangle \langle g_y , \psi_0 \rangle \langle g_y , \hat{A} g_z\rangle\vert
    \\& =\vert \textup{Pol}(y,z) \vert \vert\langle g_z , \psi_0 \rangle \langle g_y , \psi_0 \rangle \langle g_y ,  g_z\rangle\vert
\end{split}
\end{align}
and hence the potential for the Hamiltonian Monte Carlo algorithm is given by
\begin{align}
\begin{split} \label{AppendixE:Ham}
    U[\hat{A}](y,z)  = -\log{\vert\langle g_z , \psi_0 \rangle \vert } - \log{\vert\langle g_y , \psi_0 \rangle \vert } - \log{\vert\langle g_y , g_z\rangle\vert} - \log{\vert \textup{Pol}(y,z) \vert}.
\end{split}
\end{align} 
\begin{example}[$\hat{A}= \mathrm{Id}$]
In this case, $\textup{Pol}(y,z)=1$ and the potential reduces to
\begin{align}
    \begin{split}
         &U[\textup{Id}](y,z) 
         = \frac{1}{4\epsilon} \left[(z-z_0)^T  (z-z_0) + (y-z_0)^T   (y-z_0)+ (z-y)^T   (z-y)\right].
    \end{split}
\end{align}
Hence, the gradient is given by
\begin{align}
    \nabla U[\textup{Id}](y,z) = \frac{1}{2\epsilon}\begin{pmatrix}
            (y-z_0) -   (z-y)
        \\
           (z-z_0) +  (z-y)
    \end{pmatrix}.
\end{align}

\end{example}
By the polynomial decomposition \eqref{AppendixE:pol} and the Potential \eqref{AppendixE:Ham}, for an operator $\hat{A}$ it follows 
\begin{align}
    U[\hat{A}](y,z) = U[\textup{Id}](y,z) - \log{\vert \textup{Pol}(y,z) \vert}
\end{align}
and 
\begin{align}\label{EQ:HMC_gradient}
\begin{split}
    &\nabla U[\hat{A}](y,z) 
    = \nabla U[\textup{Id}](y,z)  
 \\&+\frac{1}{\vert \textup{Pol}(y,z) \vert^2} 
\left[ \textup{Re}(\textup{Pol}(y,z)) \nabla \textup{Re}(\textup{Pol}(y,z)) + \textup{Im}(\textup{Pol}(y,z))  \nabla \textup{Im}(\textup{Pol}(y,z)) \right] .
\end{split}
\end{align}
Hence to derive the gradients for further examples of $\hat{A}$, one only needs polynomial terms. Those can be taken from Appendix~\ref{Appendix:Inner_products}.

\end{document}